\documentclass[11pt,a4paper, reqno]{amsart}
\usepackage{amsmath,amsfonts,amssymb,mathrsfs}
\usepackage{mathtools}
\usepackage{graphicx}
\usepackage{comment}
\usepackage{xfrac}
\usepackage{subfigure}
\usepackage{bbm}
\usepackage{epstopdf}
\usepackage{pstricks}
\usepackage{pst-node}
\usepackage[linktocpage=true,colorlinks=true,linkcolor=blue,citecolor=red,urlcolor=magenta,pdfborder={0 0 0}]{hyperref}
\usepackage{fancyhdr}
\usepackage{mathscinet}
\usepackage{dutchcal}
\usepackage{xcolor}
\usepackage[english]{babel}
\usepackage{mathscinet}

\usepackage{chngcntr}

\usepackage{appendix}

\definecolor{forestgreen}{rgb}{0.13, 0.55, 0.13}
\definecolor{anna}{rgb}{0.01, 0.28, 1.0}

\newtheorem{theorem}{\bf Theorem}[section]
\newtheorem{lemma}[theorem]{\bf Lemma}
\newtheorem{corollary}[theorem]{\bf Corollary}
\newtheorem{definition}[theorem]{\bf Definition}
\newtheorem{remark}[theorem]{\bf Remark}
\newtheorem{proposition}[theorem]{\bf Proposition}

\newcommand{\R}{\mathbb{R}}

\newcommand{\N}{\mathbb{N}}

\newcommand{\OO}{\mathbb{O}}
\newcommand{\I}{\mathbb{I}}

\newcommand{\W}{\mathcal{W}}

\def \Z {{\mathbb{Z}}}

\def \L {\mathscr{L}}
\def \K {\mathscr{K}}

\def \o {{\omega}}
\def \a {{\alpha}}
\def \b {{\beta}}
\def \cc {{\chi}}
\def \g {{\gamma}}
\def \d {{\delta}}
\def \e {{\varepsilon}}

\def \epsilon {{\varepsilon}}

\def \k {{\kappa}}
\def \l {{\lambda}}
\def \r {{\rho}}

\def \t {{\tau}}
\def \m {{\mu}}

\def \x {{\xi}}

\def \z {{\zeta}}

\def \phi {{\varphi}}

\def \G {{\Gamma}}
\def \O {{\Omega}}
\def \D {{\Delta}}

\def \div {{\text{\rm div}}}
\def \loc {{\text{\rm loc}}}
\def \trace {{\text{\rm tr}}}

\def \diag {{\text{\rm diag}}}
\def \meas {{\text{\rm meas}}}
\def\p{\partial}
\def \tilde {\widetilde}

\def \Q {{\mathcal{Q}}}

\def \D {{\mathcal{D}}}

\usepackage{mathtools}

\begin{document}
	\title{A note on the weak regularity theory for degenerate Kolmogorov equations}
	
	\author{Francesca Anceschi}
	\address{Dipartimento di Matematica e Applicazioni
		"Renato Caccioppoli" -
		Università degli Studi di Napoli "Federico II": 
		Via Cintia, Monte S. Angelo
		I-80126 Napoli, Italy}
	\email{francesca.anceschi@unina.it}
	
	\author{Annalaura Rebucci}
	\address{Dipartimento di Scienze Matematiche, Fisiche e Informatiche - Università degli Studi di Parma: Parco Area delle Scienze, 7/A 43124 Parma, Italy}
	\email{annalaura.rebucci@unipr.it}
	
	\date{\today}

	\begin{abstract}
		\noindent
		The aim of this work is to prove a Harnack inequality and the H\"older continuity 
		for weak solutions to the Kolmogorov equation $\L u = f$ with measurable 
		coefficients, integrable lower order terms and nonzero source term. We introduce a function space $\W$, suitable for the study 
		of weak solutions to $\L u = f$, that allows us to prove a weak Poincaré inequality. 
		More precisely, our goal is to prove a weak Harnack inequality
		for non-negative super-solutions by considering their Log-transform and 
		following S. N. Kr\u uzkov (1963). Then this functional inequality is combined with a 
		classical covering argument (Ink-Spots Theorem) that we extend for the first time to the case of 
		ultraparabolic equations.
		
		
		
	\end{abstract}
	
	\maketitle

	\hypersetup{bookmarksdepth=2}
	\setcounter{tocdepth}{1}
	\tableofcontents
	\setcounter{equation}{0}\setcounter{theorem}{0}
		\section{Introduction}
	The aim of this work is to study the De Giorgi-Nash-Moser regularity theory for weak solutions to the second order partial differential equation of Kolmogorov type of the form
\begin{equation}\label{defL}
\begin{split}
\L u (x,t):&=\sum_{i,j=1}^{m_0}\partial_{x_i}\left(a_{ij}(x,t)\partial_{x_j}u(x,t)\right)+\sum_{i,j=1}^N b_{ij}x_j\partial_{x_i}u(x,t)-\partial_t u(x,t)\\
&+\sum_{i=1}^{m_0}b_i(x,t)\partial_i u(x,t)+c(x,t)u(x,t)=f,
\end{split}
\end{equation}
where $z=(x,t)=(x_1,\ldots,x_N,t)\in \R^{N+1}$ and $1 \leq m_0 \leq N$. Moreover, the matrices $A_0=(a_{ij}(x,t))_{i,j=1,\ldots,m_0}$ and $B=(b_{ij})_{i,j=1,\ldots,N}$ satisfy the following structural assumptions.

\medskip

\begin{itemize}
\item[\textbf{(H1)}] The matrix $A_0$ is symmetric with real measurable entries. Moreover, $a_{ij}(x,t)=a_{ji}(x,t)$, for every $i,j=1,\ldots,m_0$, and there exist two positive constants $\lambda$ and $\Lambda$ such that 
\begin{eqnarray}\label{hypcost}
\lambda |\xi|^2 \leq \sum_{i,j=1}^{m_0}a_{ij}(x,t)\xi_i\xi_j \leq \Lambda|\xi|^2
\end{eqnarray}
for every $(x,t) \in \R^{N+1}$ and $\xi \in \R^{m_0}$. The matrix B has constant entries.
\end{itemize}
\medskip
Note that we allow the operator $\L$ to be strongly degenerate whenever $m_0 < N$. However, the first order part of $\L$ may induce a strong regularizing property. Indeed, it is known that, under suitable assumptions on the matrix $B$, the operator $\L$ is hypoelliptic, namely that every distributional solution $u$ to $\L u = f$ defined in some open set $\Omega \subset \R^{N+1}$ belongs to $C^\infty(\Omega)$ and it is a classical solution to $\L u = f$, whenever $f \in C^\infty(\Omega)$.
 In the sequel, we will therefore rely on the following assumption.

\medskip

\begin{itemize}
\item[\textbf{(H2)}] The \textit{principal part operator $\K$ of $\L$} is hypoelliptic
and homogeneous of degree $2$ with respect to the family of dilations $\left( \d_{r} \right)_{r>0}$ introduced in \eqref{gdil}, where $\K$ is defined as
\begin{eqnarray}\label{defK}
	\K u(x,t):=\sum_{i=1}^{m_0}\partial^2_{x_i} u(x,t) + \sum_{i,j=1}^N b_{ij}x_j\partial_{x_i}u(x,t)-\partial_t u(x,t),
	\qquad (x,t) \in \R^{N+1}.
\end{eqnarray}
\end{itemize}

\noindent
 It is clear that if $\L$ is uniformly parabolic (i.e. $m_0=N$ and $B\equiv \mathbb{O}$), then assumption $\textbf{(H2)}$ is satisfied. In fact, in this case the principal part operator $\K$ is simply the heat operator, which is known to be hypoelliptic. For further information on the hypoellipticity of $\K$ and on the equivalent structural condition for the matrix B we refer to Section \ref{preliminaries}.

\medskip

In the sequel we will also make use of the following notation in order to introduce a compact formulation for the operator  $\L$.
More precisely, here and in the sequel
\begin{eqnarray*}
	D=(\partial_{x_1},\ldots,\partial_{x_N}),\quad \langle\cdot,\cdot \rangle, \quad \div
\end{eqnarray*}
respectively denote the gradient, the inner product and the divergence in $\R^N$. In addition, 
\begin{equation*}
	D_{m_0}=(\partial_{x_1},\ldots,\partial_{x_{m_0}}), \quad \div_{m_0}
\end{equation*}
denote the partial gradient and the partial divergence in the first $m_0$ components, respectively. Moreover, we introduce the matrix 
\begin{eqnarray*}
A(x,t)=\left(a_{ij}(x,t)\right)_{1\leq i,j \leq N},
\end{eqnarray*}
where $a_{ij}$, for every $i,j=1,\ldots,m_0$, are the coefficients appearing in \eqref{defL}, while $a_{ij}\equiv 0$ whenever $i > m_0$ or $j>m_0$, and we let
\begin{align}  \label{drift}
b(x,t):=\left(b_1(x,t),\ldots,b_{m_0}(x,t),0,\ldots,0\right),
 \qquad Y:=\sum_{i,j=1}^N b_{ij}x_j\partial_{x_i}u(x,t)-\partial_t u(x,t).
\end{align}
Now, we are in a position to rewrite the operator $\L$ in the following compact form
\begin{eqnarray*}
\L u=\div(A D u) + Yu + \langle b, D u\rangle+ cu
\end{eqnarray*}
and to state our last assumption on the integrability of $b$, $c$ and of the source term $f$.

\medskip

\begin{itemize}
\item[\textbf{(H3)}] $c, f \in L^q(\O)$ and $b \in \left( L^q (\O) \right)^{m_0}$
	for some $q >\frac{3}{4}\left( Q+2\right)$. Moreover, we assume
	\begin{equation*}
		\div b \ge 0  \qquad \text{in} \, \, \O.
	\end{equation*}
\end{itemize}
\noindent
The physical interpretation of the sign of the divergence of $b$ can be understood by considering the Vlasov-Poisson-Fokker-Planck equation \cite{vlasov}, for which the lower order term $b$ represents the electrostatic or gravitational forces.  The equations with the term $b$ satisfying the structural assumption $\div b \geq 0$ arise also in some other applications, like the ones contained in \cite{S-div, div}. Moreover, the sign assumption on the divergence of $b$ is also quite relevant in the case of parabolic equations, since it has several applications to, for instance, incompressible flows and magnetostrophic turbulence models for the Earth’s fluid core, e.g. \cite{Moffatt}. In particular, it is nowadays known that the sign (or either the divergence free, i.e. $\div b=0$) assumption can be considered to relax the regularity assumptions on $b$ under which one can prove the Harnack inequality and other results, see \cite{S-div, Ignatova}. Nevertheless, in our case as in the parabolic setting presented in \cite{S-div}, there is still the need to require that the divergence of $b$ exists in the sense of distributions and that $b$ is at least locally integrable up to a certain power.

\subsection{Main results}
Our aim is to prove the local boundedness and the local H\"older continuity, alongside with a Harnack inequality, for weak solutions to $\L u = f$ under the assumptions \textbf{(H1)-(H3)}.
In particular, in order to  expose our main results we first need to introduce some preliminary notation. From now on, we consider a set $\O = \O_{m_{0}} \times \O_{N - m_{0} + 1}$ of $\R^{N+1}$, where
$\O_{m_{0}}$ is a bounded Lipschitz domain of $\R^{m_{0}}$ and 
$\O_{N-m_{0} + 1}$ is a bounded Lipschitz domain of $\R^{N - m_{0} + 1}$. 
This is not restrictive since the cylinders $\Q$ introduced in \eqref{rcylind} that we consider 
in our local analysis satisfy the Lipschitz boundary assumption.  
Then we split the coordinate $x\in\R^N$ as
\begin{equation}\label{split.coord.RN}
	x=\big(x^{(0)},x^{(1)},\ldots, x^{(\kappa)}\big), \qquad x^{(0)}\!\in\R^{m_0}, \quad x^{(j)}\!\in\R^{m_j}, \quad j\in\{1,\ldots,\kappa\},
\end{equation}
where we have that  in accordance with the scaling of the differential equation (see \eqref{gdil} below)
every $m_{j}$ is a positive integer such that
\begin{equation*}
		\sum \limits_{j=1}^{\kappa} m_j = N 
		\qquad \text{and} \qquad N \ge m_0 \ge m_1
		\ge \ldots \ge m_\k \ge 1  .
\end{equation*}
We denote by $\D(\O)$ the set of $C^\infty$ functions compactly supported in $\O$ and by $\D'(\O)$ the set of distributions in $\O$. From now on, $H^{1}_{x^{(0)}}$ denotes the Sobolev space of functions $u \in  L^{2} (\O_{m_{0}})$ with 
distribution gradient $D_{m_0}u$ lying in $( L^{2} (\O_{m_{0}}) )^{m_{0}}$, i.e. 
\begin{equation*}
	H^{1}_{x^{(0)}} := \left\{ u \in L^{2} (\O_{m_{0}}) : \, D_{m_{0}} u \in ( L^{2} (\O_{m_{0}}) )^{m_{0}}
	\right\},
\end{equation*}
and we set 
\begin{equation*}
	\| u \|_{H^{1}_{x^{(0)}}} := \| u \|_{L^{2} (\O_{m_{0}}) } + \| D_{m_{0} } u \|_{L^{2} (\O_{m_{0}})}.
\end{equation*}
We let $H^{1}_{c,x^{(0)}}$ denote the closure of $C^\infty_{c}(\O_{m_{0}})$ in the norm of $H^{1}_{x^{(0)}}$ and 
we recall that $C^\infty_{c}(\overline \O_{m_{0}})$ is dense in $H^{1}_{x^{(0)}}$ since $\O_{m_{0}}$ is a bounded Lipschitz domain by assumption. Moreover, $H^{1}_{c,x^{(0)}}$ is a reflexive Hilbert space and thus we may consider its dual space 
$$
\left( H^{1}_{c,x^{(0)}} \right)^{*} = H^{-1}_{x^{(0)}} \quad \text{and} \quad 
\left( H^{-1}_{x^{(0)}} \right)^{*} = H^{1}_{c,x^{(0)}},
$$
where the notation we consider is the classical one. Hence, from now on we denote by $H^{-1}_{x^{(0)}}$ the dual of 
$H^{1}_{c,x^{(0)}}$ acting on functions in $H^{1}_{c,x^{(0)}}$ through the duality pairing 
$\langle \cdot , \cdot \rangle := \langle \cdot , \cdot \rangle_{H^{1}_{x^{(0)}},H^{1}_{c,x^{(0)}}} $.
In a standard manner, see for instance \cite{AM, LN}, we let $\W$ denote the closure of $C^{\infty} (\overline \O)$ in the 
norm 
\begin{equation}
	\label{normW}
	\| u \|^2_{\W} = \| u \|^2_{L^2(\O_{N-m_{0} + 1};H^1_{x^{(0)}})} + \| Y u \|^2_{L^2(\O_{N-m_{0} + 1};H^{-1}_{x^{(0)}})},
\end{equation}
where the previous norm can explicitly computed as follows:
\begin{align*}
	\| u \|^2_{\W} = \left( \int_{\O_{N-m_{0} + 1}} \| u(\cdot, y,t) \|_{H^1_{x^{(0)}}}^{2} dy \, dt \right)^{\frac12} + 
				 \left( \int_{\O_{N-m_{0} + 1}}  \| Y u (\cdot, y,t) \|^2_{H^{-1}_{x^{(0)}}} dy \, dt \right)^{\frac12},
\end{align*}
where $y = (x^{(1)},\ldots, x^{(\kappa)})$. In particular, $\W$ is a Banach space and we remark that the dual of
$L^{2}(\O_{N-m_{0}+1}; H^{1}_{c,x^{(0)}})$ satisfies 
\begin{align*}
	&\left( L^{2}(\O_{N-m_{0}+1}; H^{1}_{c,x^{(0)}}) \right)^{*} =  L^{2}(\O_{N-m_{0}+1}; H^{-1}_{c,x^{(0)}}) 
	\quad \text{and} \quad  \\
	&\qquad \qquad  \qquad \qquad \left( L^{2}(\O_{N-m_{0}+1}; H^{-1}_{c,x^{(0)}})   \right)^{*} = L^{2}(\O_{N-m_{0}+1}; H^{1}_{c,x^{(0)}}) .
\end{align*}
From now on, we consider the shorthand notation $L^{2}H^{-1}$ to denote $L^{2}(\O_{N-m_{0}+1}; H^{-1}_{c,x^{(0)}})$.

\medskip

The space of functions $\W$ is the most natural framework for the study of the weak regularity theory for the operator $\L$. In particular, it is an
extension of the functional setting proposed by Armstrong and Mourrat in \cite{AM} for the study of the kinetic Kolmogorov-Fokker-Planck equation, where the authors show that it is sufficient to control $u \in L^2$ and $Y u \in L^2H^{-1}$ in order to derive new Poincaré inequalities, such as Proposition \ref{weak-poincare}, on which we base our analysis later on.
Moreover, we refer to \cite[Section 2]{LN} for some properties of the space $\W$. Lastly, we remark that 
the major issue when dealing with the space $\W$ is that it requires to handle 
the duality pairing between $L^{2}H^{1}$ and $L^{2}H^{-1}$. To this end, we take advantage of 
the following remark, see \cite[Chapter 4]{SK}. 
\begin{remark}
\label{remark-Y}
For every open subset $A \subset \R^{n}$ and for every 
function $g \in H^{-1}(A )$ there exist two functions
$H_0$, $H_1 \in L^2(A)$ such that 
\begin{equation*}
	g = \div_{m_{0}} H_{1} + H_{0} \qquad \text{and} \qquad 
	\Vert H_0 \Vert_{L^2(A)}+\Vert H_1 \Vert_{L^2(A)} \leq 
	2\Vert g \Vert_{H^{-1}(A)}.
\end{equation*}
\end{remark}

Now, we introduce the definition of weak solution we consider in our work. 
\begin{definition}\label{weak-sol2}
A function $u \in \W$ is a weak solution to \eqref{defL} with source term $f \in L^2(\O)$ if for every non-negative test function $\phi \in \D(\O)$, we have
\begin{align}\label{kolmo}
	   \int_{\O} - \langle A Du, D\phi \rangle + \phi Y u + \langle b , Du \rangle \phi + c u \phi= \int_{\O} f \phi.
	\end{align}
In the sequel, we will also consider weak sub-solutions to \eqref{defL}, namely functions $u \in \W$ that satisfy the following inequality
\begin{align}\label{kolmo-sub}
	   \int_{\O} - \langle A Du, D\phi \rangle + \phi Y u + \langle b , Du \rangle \phi + c u \phi \geq \int_{\O} f \phi,
	\end{align}
	for every non-negative test function $\phi \in \D(\O)$. A function $u$ is a super-solution to \eqref{defL} if $-u$ is a sub-solution.
\end{definition}
This framework is quite classical for the study of the weak regularity theory for the kinetic Kolmogorov-Fokker-Planck equation \cite{AM, GIMV, GI, GM}, that can be recovered from \eqref{defL} by choosing $N=2d$, $\k = 1$, $m_0 = m_1= d$ and $c\equiv 0$. Still, even if it is the most natural framework for \eqref{defL}, to the best of our knowledge it has never been considered in literature yet. Indeed, the weak regularity theory for the operator 
$\L$ has been widely developed during the last decade starting from the paper \cite{PP} by Pascucci and Polidoro, where the authors worked under the stronger assumption
$Y u \in L^2_{\loc}(\O)$. In the same framework (i.e. $Y u \in L^2(\O)$), Wang and Zang \cite{WZ4, WZ3} lately proved the results for equation \eqref{defL} with $a=b=f=0$ on the local H\"older continuity for solutions. 

\medskip 

The aim of this paper is to prove the local H\"older continuity and a Harnack inequality for solutions to \eqref{defL} in the sense of Definition \ref{weak-sol2}. Our method is based on the combination of three fundamental ingredients - boundedness of weak solutions, weak Poincaré inequality and Log-transformation - in the same spirit of the recent paper \cite{GI} for the Fokker-Planck equation. In particular, we give an answer to Remark $4$, p. 2 of \cite{GI}
and we are also able to simplify the proof proposed in \cite{WZ-preprint} to obtain the local H\"older continuity for weak solutions to \eqref{defL}. 
We classically reduce the local study to the case where $\Q^0$ is at unit scale and for some reasons we expose below in Section \ref{harnack}, $\Q^0$
takes the form $B_{R_0}\times B_{R_0}\times \ldots \times B_{R_0} \times(-1,0] $
for some large constant $R_0$ only
depending on the dimension $Q$ and the ellipticity constants $\lambda, \Lambda$ in \textbf{(H1)}.

As we will see in Section \ref{preliminaries}, the suitable geometry when dealing with operator $\L$ is given by an homogeneous Lie group structure defined on $\mathbb{R}^{N+1}$. Our results naturally reflect this non-Euclidean setting. Let $"\circ"$ denote the composition law introduced in \eqref{grouplaw} and $\left(\delta_r\right)_{r>0}$ the family of dilations defined in \eqref{gdil}. We consider the unit past cylinder
\begin{eqnarray}\label{unitcylind}
\Q_1 := B_1 \times B_1 \times \ldots \times B_1 \times (-1,0),
\end{eqnarray}
defined through the open balls 
\begin{eqnarray}\label{openball}
B_1 = \lbrace x^{(j)}\!\in\R^{m_j} : \vert x \vert \leq 1 \rbrace,
\end{eqnarray}
where $j=0,\ldots,\kappa$ and $\vert \cdot \vert$ denotes the euclidean norm in $\R^{m_j}$. Now, for every $z_0 \in \mathbb{R}^{N+1}$ and $r>0$, we set
\begin{eqnarray}\label{rcylind}
\Q_r(z_0):=z_0\circ\left(\delta_r\left(\Q_1\right)\right)=\lbrace z \in \mathbb{R}^{N+1} \,:\, z=z_0\circ \delta_r(\zeta), \zeta \in \Q_1 \rbrace.
\end{eqnarray}
We are now in a position to state our main results.

\begin{theorem}[Weak Harnack inequality]
	\label{weak-harnack}
Let $\Q^0=B_{R_0}\times B_{R_0}\times \ldots \times B_{R_0} \times(-1,0] $ and let $u$ be a non-negative weak super-solution to $\L u = f$ in $\O\supset \Q^0$ under the 
	assumptions \textbf{(H1)}-\textbf{(H3)}. Then we have
\begin{eqnarray}
\left(\int_{\Q_-}u^p\right)^{\frac{1}{p}}\leq C\left(\inf_{\Q_+} u + \Vert f \Vert_{L^q(\Q^0)}\right),
\end{eqnarray}
where $\Q_+=B_{\omega}\times B_{\omega^3}\times\ldots \times B_{\omega^{2\kappa+1}}\times (-\omega^2,0]$ and $\Q_-=B_{\omega}\times B_{\omega^3}\times\ldots \times B_{\omega^{2\kappa+1}}\times (-1,-1+\omega^2]$. Moreover, the constants $C$, $p$, $\omega$ and $R_0$ only depend on the homogeneous dimension $Q$ defined in \eqref{hom-dim}, $q$ and on the ellipticity constants $\lambda$ and $\Lambda$ in \eqref{hypcost}. 
\end{theorem}
We remark that as in \cite{PP}, the radius $\o$ is small enough so that when “stacking cylinders”
over a small initial one contained in $\Q^-$, the cylinder $\Q^+$ is captured, see Lemma \ref{stackedcylinders}. As far as we are concerned with $R_0$, it is large enough so that it is possible to apply the expansion of positivity lemma (see Lemma \ref{expand-pos}) to every stacked cylinder. By combining the local boundedness of positive weak sub-solutions proved in Theorem \ref{boundedness} and Theorem \ref{weak-harnack} we obtain the following Harnack inequality, 
which is an extension of the analogous result contained in \cite{AEP, GIMV}.
\begin{theorem}[Harnack inequality]
	\label{harnack-thm}
	Let $\Q^0=B_{R_0}\times B_{R_0}\times \ldots \times B_{R_0} \times(-1,0] $ 
	and let $u$ be a non-negative weak solution to $\L u = f$ in 
	$\O\supset \Q^0$ under the assumptions \textbf{(H1)}-\textbf{(H3)}. 
	Then	 we have
	\begin{eqnarray}
		\sup \limits_{\Q_{-}} u \, \le \, C \left( \inf_{\Q_+} u + \Vert f \Vert_{L^q(\Q^{0})} \right),
	\end{eqnarray}
	where $\Q_+=B_{\omega}\times B_{\omega^3}\times\ldots \times B_{\omega^{2\kappa+1}}\times (-\omega^2,0]$
	and $\Q_-=B_{\omega}\times B_{\omega^3}\times\ldots \times B_{\omega^{2\kappa+1}}\times (-1,-1+\omega^2]$. 
	Moreover, the constants $C$, $\omega$ and $R_0$ only depend on the homogeneous dimension $Q$ defined in 
	\eqref{hom-dim}, $q$ and on the ellipticity constants $\lambda$ and $\Lambda$ in \eqref{hypcost}. 
\end{theorem}
Since our proof of Theorem \ref{weak-harnack} is constructive, as it is based on the combination of an expansion of positivity argument with the weak Poincaré inequality of Proposition \ref{weak-poincare}, also the proof of Theorem \ref{harnack-thm} is constructive. 
Moreover, the weak Harnack inequality also implies the H\"older regularity of weak solutions. In order to state this result, we first need to give the following definition.
\begin{definition}
    \label{holdercontinuous}
    Let $\a$ be a positive constant, $\a \le 1$, and let $\O$ be an open subset of $\R^{N+1}$. We 
    say that a function $f : \O \longrightarrow \R$ is H\"older continuous with exponent $\a$ in $\O$
    with respect to the group $\mathbb{K}=(\mathbb{R}^{N+1},\circ)$, defined in \eqref{grouplaw}, (in short: H\"older 
    continuous with exponent $\a$, $f \in C^\a_{K} (\O)$) if there exists a positive constant $C>0$ such that 
    \begin{equation*}
        | f(z) - f(\z) | \le C \; d(z,\zeta)^{\a} \qquad { \rm for \, every \, } z, \z \in \O,
    \end{equation*}
where $d$ is the distance defined in \eqref{def-dist}.

To every bounded function $f \in C^\a_{K} (\O)$ we associate the semi-norm
    	\begin{equation*}
       		[ f ]_{C^{\a} (\O)} = 
       		\hspace{1mm} 
       		 \sup \limits_{z, \z \in \O \atop  z \ne \z} \frac{|f(z) - f(\z)|}{d(z,\zeta)^{\a}}.
        \end{equation*}
    Moreover, we say a function $f$ is locally H\"older continuous, and we write $f \in C^{\a}_{K,\loc}(\O)$,
    if $f \in C^{\a}_{K}(\O')$ for every compact subset $\O'$ of $\O$.
\end{definition}
We are now in a position to state the following result.

\begin{theorem}[H\"older regularity]
	\label{local-holder}
	There exists $\a \in (0,1)$ only depending on dimension $Q$, $\l$, $\Lambda$ such that all weak solutions 
	$u$ to \eqref{defL} under assumption \textbf{(H1)}- \textbf{(H3)} in $\O \supset \Q_{1}$ satisfy
	\begin{equation*}
		[ u ]_{C^{\a} (Q_{\frac12})} \, \le C \left( \| u \|_{L^{2}(\Q_{1})} + \| f \|_{L^{q}(\Q_{1})} \right),
	\end{equation*}
	where the constant $C$ only depends on the homogeneous dimension $Q$ defined in \eqref{hom-dim}, $q$
	and the ellipticity constants $\l$, $\Lambda$.
\end{theorem}
\begin{remark}\label{remarkWZ}
	The estimates presented in Theorem \ref{harnack-thm} and Theorem \ref{local-holder}
	can be scaled and stated in arbitrary cylinders thanks to the invariance of the operator $\L$ 
	with respect to the group operations (dilations \eqref{gdil} and translations \eqref{grouplaw}) introduced in Section 
	\ref{preliminaries}. Moreover, all of our main results still hold true if we replace assumption \textbf{(H3)} with the 
	following one:
	\begin{center}
	$c, f \in L^q(\O)$ for some $q >\frac{Q+2}{2}$and 
	$b \in \left( L^q (\O) \right)^{m_0}$ for some $q >Q+2$ in $\O$.
	\end{center}
	This is exactly the integrability requirement assumed by Wang and Zhang in \cite{WZ-preprint}, thus our work is an improvement of this article for the case homogeneous ultraparabolic operators. Moreover, as we work in the appriopriate functional setting $\W$, our proofs of the main results are simpler than the ones proposed in \cite{WZ-preprint}. 
\end{remark}

\subsection{Comparison with existing results and organization of the paper}
The purpose of this paper is to provide a complete characterization of the De Giorgi-Nash-Moser weak regularity theory for the 
Kolmogorov equation in divergence form with measurable coefficients in the more natural space $\W$ for weak solutions to $\L u 
= f$. 
First of all, 
we improve the existing result proved by Wang and Zhang in \cite{WZ4, WZ3, WZ-preprint} by providing a 
simpler and more elegant proof of the weak Harnack inequality and by explicitly showing the derivation of an invariant Harnack 
inequality, the first one available for weak solutions to $\L u = f$. In particular, our results hold true also under the 
assumptions of \cite{WZ-preprint} as it is pointed out in Remark \ref{remarkWZ} and Remark \ref{remark-boundedness}, and 
extend the ones obtained in the particular case of the kinetic Fokker-Planck equation in \cite{GI, GIMV, LN}.

Our proof is based on the combination of a weak Poincaré inequality for the space $\W$
with the Ink-spots theorem in the same spirit of \cite{IS-weak, GI}. We remark that, up to this day, the Ink-spot theorem has only been available just for the Fokker-Planck equation case in $\R^{2n+1}$, and thus also the technical result that we prove in Appendix \ref{ink-spot} is a complete novelty in this more general framework. Moreover, we manage to lower the integrability assumption on the lower order coefficients and on the source term $f$ and to extend the Moser's iterative method proposed in \cite{PP} to this more general case. 

The structure of the paper is the following. In Section \ref{preliminaries} we recall some known facts about operators $\L$ and we state some preliminary results. The proofs of some intermediate theorems (a Sobolev type and a Caccioppoli type inequality), together with the Moser's iterative method, are presented in Section \ref{moser}. Section \ref{poincare} is devoted to the proof of a weak Poincar\`{e} inequality. In Section \ref{harnack} we derive the weak Harnack inequality by combining the expansion of positivity and a covering argument known as Ink-spots theorem, whose proof is contained in Appendix \ref{ink-spot}. Finally, in Appendix \ref{stacked} we state a technical lemma regarding stacked cylinders.

\subsection*{Acknowledgements}
The authors are grateful to Prof. Sergio Polidoro for suggesting the question and the fruitful discussions. 
The first author is funded by the research grant PRIN2017 2017AYM8XW ``Nonlinear Differential Problems via Variational, Topological and Set-valued Methods''.

\setcounter{equation}{0}\setcounter{theorem}{0}
\section{Preliminaries}
\label{preliminaries}
In this Section we recall notation and results we need in order to deal with the non-Euclidean 
geometry underlying the operators $\L$ and $\K$. We refer to the articles \cite{PP,LP} and to the survey \cite{APsurvey} for a comprehensive treatment of this subject.

As first observed by Lanconelli and Polidoro in \cite{LP}, the operator $\K$ is invariant with respect to left 
translation in the group $\mathbb{K}=(\mathbb{R}^{N+1},\circ)$, where the group law is defined by 
\begin{equation}
	\label{grouplaw}
	(x,t) \circ (\xi, \tau) = (\xi + E(\tau) x, t + \tau ), \hspace{5mm} (x,t),
	(\xi, \tau) \in \R^{N+1},
\end{equation}
and
\begin{equation}\label{exp}
	E(s) = \exp (-s B), \qquad s \in \R.
\end{equation} 
Then $\mathbb{K}$ is a non-commutative group with zero element $(0,0)$ and inverse
\begin{equation*}
(x,t)^{-1} = (-E(-t)x,-t).
\end{equation*}
For a given $\zeta \in \R^{N+1}$ we denote by $\ell_{\z}$ the left traslation on $\mathbb{K}=(\R^{N+1},\circ)$ defined as follows
\begin{equation*}
	\ell_{\z}: \R^{N+1} \rightarrow \R^{N+1}, \quad \ell_{\z} (z) = \z \circ z.
\end{equation*}
Then the operator $\K$ is left invariant with respect to the Lie product $\circ$, that is
\begin{equation*}
	\label{ell}
    \K \circ \ell_{\z} = \ell_{\z} \circ \K \qquad {\rm \textit{or, equivalently,}} 
    \qquad \K\left( u( \z \circ z) \right)  = \left( \K u \right) \left( \z \circ z \right),
\end{equation*}
for every $u$ sufficiently smooth.

We recall that, by \cite{LP} (Propositions 2.1 and 2.2), assumption \textbf{(H2)}
is equivalent to assume that, for some basis on $\R^N$, the matrix $B$ takes the following form
\begin{equation}
	\label{B}
	B =
	\begin{pmatrix}
		\OO   &   \OO   & \ldots &    \OO   &   \OO   \\  
		B_1   &    \OO  & \ldots &    \OO    &   \OO  \\
		\OO    &    B_2  & \ldots &  \OO    &   \OO   \\
		\vdots & \vdots & \ddots & \vdots & \vdots \\
		\OO    &  \OO    &    \ldots & B_\k    & \OO
	\end{pmatrix}
\end{equation}
where every $B_k$ is a $m_{k} \times m_{k-1}$ matrix of rank $m_j$, $j = 1, 2, \ldots, \k$ 
with 
\begin{equation*}
	m_0 \ge m_1 \ge \ldots \ge m_\k \ge 1 \hspace{5mm} \text{and} \hspace{5mm} 
	\sum \limits_{j=0}^\k m_j = N.
\end{equation*} 
In the sequel we will assume that $B$ has the canonical form \eqref{B}.
We remark that assumption \textbf{(H2)} is implied by the condition introduced by H\"{o}rmander in \cite{H}:
\begin{equation}\label{e-Horm}
{\rm rank\ Lie}\left(\partial_{x_1},\dots,\partial_{x_{m_0}},Y\right)(x,t) =N+1,\qquad \forall \, (x,t) \in \R^{N+1},
\end{equation}
where ${\rm  Lie}\left(\partial_{x_1},\dots,\partial_{x_{m_0}},Y\right)$ denotes the Lie algebra generated by the first order differential operators $\left(\partial_{x_1},\dots,\partial_{x_{m_0}},Y\right)$ computed at $(x,t)$.
Yet another condition equivalent to {\bf (H2)}, (see \cite{LP}, Proposition A.1), is that
\begin{eqnarray}\label{Cpositive}
C(t) > 0, \quad \text{for every $t>0$},
\end{eqnarray}
where
\begin{equation*}
	C(t) = \int_0^t \hspace{1mm} E(s) \, A_0 \, E^T(s) \, ds,
\end{equation*}
and $E(\cdot)$ is the matrix defined in \eqref{exp}.

We now recall that, under the the hypothesis of hypoellipticity, H\"{o}rmander in \cite{H} constructed the fundamental solution of $\K$ as
\begin{equation} \label{eq-Gamma0}
	\Gamma (z,\zeta) = \Gamma (\zeta^{-1} \circ z, 0 ), \hspace{4mm} 
	\forall z, \zeta \in \R^{N+1}, \hspace{1mm} z \ne \zeta,
\end{equation}
where
\begin{equation} \label{eq-Gamma0-b}
	\Gamma ( (x,t), (0,0)) = \begin{cases}
		\frac{(4 \pi)^{-\frac{N}{2}}}{\sqrt{\text{det} C(t)}} \exp \left( - 
		\frac{1}{4} \langle C^{-1} (t) x, x \rangle - t \, \trace (B) \right), \hspace{3mm} & 
		\text{if} \hspace{1mm} t > 0, \\
		0, & \text{if} \hspace{1mm} t < 0.
	\end{cases}
\end{equation}
Note that condition \eqref{Cpositive} implies that $\Gamma$ in \eqref{eq-Gamma0-b} is well-defined.

Let us now consider the second part of assumption \textbf{(H2)}. 
We say that $\K$ is invariant with respect to $(\delta_r)_{r>0}$ if  
\begin{equation}
	\label{Ginv}
      	 \K \left( u \circ \delta_r \right) = r^2 \delta_r \left( \K u \right), \quad \text{for every} \quad r>0,
\end{equation}
for every function $u$ sufficiently smooth. It is known (see Proposition 2.2 of \cite{LP}) that this dilation invariance property can be read in the expression of the matrix $B$ in \eqref{B}. More precisely, $\K$ satisfies \eqref{Ginv} if and only if the matrix $B$ takes the form \eqref{B}.
In this case, we have 
\begin{equation}
	\label{gdil}
	\d_{r} = \text{diag} ( r \I_{m_0}, r^3 \I_{m_1}, \ldots, r^{2\k+1} \I_{m_\k}, 
	r^2), \qquad \qquad r > 0.
\end{equation}
We next introduce a homogeneous norm of degree $1$ with respect to the dilations 
$(\d_{r})_{r>0}$ and a corresponding quasi-distance which is invariant with respect to the group operation \eqref{grouplaw}.
\begin{definition}
	\label{hom-norm}
	Let $\a_1, \ldots, \a_N$ be positive integers such that
	\begin{equation}\label{alphaj}
		\diag \left( r^{\a_1}, \ldots, r^{\a_N}, r^2 \right) = \d_r.
	\end{equation}
	If $\Vert z \Vert = 0$ we set $z=0$ while, if $z \in \R^{N+1} 
	\setminus \{ 0 \}$ we define $\Vert z \Vert = r$ where $r$ is the 
	unique positive solution to the equation
	\begin{equation*}
		\frac{x_1^2}{r^{2 \a_1}} + \frac{x_2^2}{r^{2 \a_2}} + \ldots 
		+ \frac{x_N^2}{r^{2 \a_N}} + \frac{t^2}{r^4} = 1.
	\end{equation*}
	Accordingly, we define the quasi-distance $d$ by
	\begin{equation}\label{def-dist}
		d(z, w) = \Vert z^{-1} \circ w \Vert , \hspace{5mm} z, w \in 
		\R^{N+1}.
	\end{equation}
\end{definition}
\begin{remark}
	As $\det E(t) = e^{t \hspace{1mm} \text{\rm trace} \,
		B} = 1$, the Lebesgue measure is invariant with respect to the translation group 
	associated to $\K$. Moreover, since $\det \d_r = r^{Q+2}$, we also have 
	\begin{equation*}
		\meas \left( \Q_r(z_0) \right) = r^{Q+2} \meas \left( \Q_1(z_0) \right), \qquad \forall
		\ r > 0, z_0 \in \R^{N+1},
	\end{equation*}
	where
	\begin{equation}
		\label{hom-dim}
		Q = m_0 + 3 m_1  + \ldots + (2\k+1) m_\k.
	\end{equation}
	The natural number $Q+2$ is called the \textit{homogeneous dimension of} $\R^{N+1}$ \textit{with respect to} 
	$(\d_{r})_{r > 0}$ and the spatial homogeneous dimension $Q$ is the hypoelliptic counterpart of the space dimension $N$ usually considered in the 
	parabolic setting, see \cite{NazarovUraltzeva}. This denomination is proper since the Jacobian determinant of $\d_{r}$ equals to $r^{Q+2}$.
\end{remark}
\begin{remark}
	The semi-norm $\Vert \cdot \Vert$ is homogeneous of degree $1$ with respect to $(\d_r )_{r>0}$, i.e.
	\begin{equation*}
		\Vert \d_r (x,t) \Vert = r  \Vert  (x,t) \Vert
		\qquad \forall r >0 \hspace{2mm} \text{and} \hspace{2mm} (x,t) \in \R^{N+1}.
	\end{equation*}
	Since in $\R^{N+1}$ all the norms which are $1$-homogeneous with respect to 
	$(\d_r)_{r>0}$ are equivalent, the norm introduced in Definition 
	\ref{hom-norm} is equivalent to the following one
	\begin{equation*}
		\Vert (x,t) \Vert_1 =|t|^{\frac{1}{2}}+|x|, \quad |x|=\sum_{j=1}^N |x_j|^{\frac{1}{\alpha_j}}
	\end{equation*}
	where the exponents $\alpha_j$, for $j=1,\ldots,N$ were introduced in \eqref{alphaj}. We 
	prefer the norm of Definition \ref{hom-norm} to $\parallel \cdot \parallel_1$ 
	because its level sets (spheres) are smooth surfaces.
\end{remark}
\begin{remark}
	\label{ball-cil} 
	We recall that for every cylinder $\Q_{r}(z_{0})$ defined in \eqref{rcylind} there exists a positive constant 
	$\overline c$ \cite[equation (21)]{WZ3} such that
	\begin{align*}
		 B_{r_{}}(x^{(0)}_{0}) &\times B_{r_{1}^{3}}(x^{(1)}_{0}) 
		\times \ldots \times
		B_{r_{1}^{2\k+1}}(x^{(\k)}_{0}) \times (-r_{1}^{2}, 0] \\
		&\subset
		\Q_{r}(z_{0}) 
		\subset B_{r_{2}}(x^{(0)}_{0}) \times B_{r_{2}^{3}}(x^{(1)}_{0}) 
		\times \ldots \times
		B_{r_{2}^{2\k+1}}(x^{(\k)}_{0}) \times (-r_{2}^{2}, 0],
	\end{align*}
	where $r_{1}= r/\overline c$ and $r_{2} = \overline c r$.
	From now on, by abuse of notation we will sometimes consider the newly introduced ball representation instead of the definition of cylinder in \eqref{rcylind}.
\end{remark}
Since $\K$ is dilation invariant with respect to $(\d_r )_{r>0}$, also its fundamental solution $\G$ is a homogeneous function of degree $- Q$, namely
\begin{equation*}
	\Gamma \left( \d_{r}(z), 0 \right) = r^{-Q} \hspace{1mm} \Gamma
	\left( z, 0 \right), \hspace{5mm} \forall z \in \R^{N+1} \setminus
	\{ 0 \}, \hspace{1mm} r > 0.
\end{equation*}
This property implies a $L^p$ estimate for Newtonian potential (c. f. for instance \cite{FO}).
\begin{theorem}
	\label{FSom}
	Let $\a \in ]0, Q+2[$ and let $G \in C (\R^{N+1} \setminus \{ 0 \})$ be a 
	$\d_\l-$homogeneous function of degree $\a - Q - 2$. If $f \in L^p 
	(\R^{N+1})$ for some $p \in ]1, +\infty[$, then the function
	\begin{equation*}
		G_f (z) := \int_{\R^{N+1}} G ( \z^{-1} \circ z) f(\z) d\z,
	\end{equation*}
	is defined almost everywhere and there exists a constant $c = c(Q, p)$ such
	that 
	\begin{equation*}
		\Vert G_f \Vert_{L^q(\R^{N+1}} \le c \max_{\Vert z 
			\Vert = 1} |G(z)| \Vert f \Vert_{L^p (\R^{N+1})},
	\end{equation*}
	where $q$ is defined by
	\begin{equation*}
		\frac{1}{q} = \frac{1}{p} - \frac{\a}{Q+2}.
	\end{equation*}
\end{theorem}
Now, we are able to define the $\G-$\textit{potential} of the function $f \in L^1(\R^{N+1})$ 
as follows
\begin{equation}
	\label{L0p}
	\G (f) (z) = \int_{\R^{N+1}} \G (z, \z ) f(\z) d\z, \qquad 
	z \in \R^{N+1}.
\end{equation}
We also remark that the potential $\G (D_{m_0} f): \R^{N+1} \longrightarrow \R^{m_0}$ is 
well-defined for any $f \in L^p (\R^{N+1})$, at least in the distributional sense, that is 
\begin{equation*}
	\label{pot}
	\G (D_{m_0} f ) (z) := - \int_{\R^{N+1}} D^{(\x)}_{m_0} \G (z, \x) \,
	f(\x) \, d\x,
\end{equation*}
where $D^{(\x)}_{m_0} \G (x,t, \x, \t)$ is the gradient with respect to $\x_1, 
\ldots, \x_{m_0}$. Based on Theorem \ref{FSom}, we derive the following explicit potential estimates by substituting $\a = 1$ and $\a=2$ when considering the $\Gamma$-potential for $f$ and $D_{0}f$, respectively. For the proof of this corollary we refer to \cite[Corollary 2.2]{PP}. 
\begin{corollary}
	\label{corollary}
	Let $f \in L^p (\Q_r)$. There exists a positive constant $c = c(T,B)$ 
	such that
	\begin{align} 
		\Vert \G (f) \Vert_{L^{p**}(\Q_r)} &\le c \Vert f 
		\Vert_{L^{p}(\Q_r)},  \label{stima1}  \\ 
		\Vert \G (D_{m_0} f) \Vert_{L^{p*}(\Q_r)} &\le c 
		\Vert f \Vert_{L^{p}(\Q_r)},\label{stima2}
	\end{align}
	where $\frac{1}{p*}= \frac{1}{p} - \frac{1}{Q+2}$ and $\frac{1}{p**}= 
	\frac{1}{p} - \frac{2}{Q+2}$.
\end{corollary}

Lastly, we show that it is possible to use the fundamental solution $\G$ as a test function in the definition of 
sub and super-solution. The following result extends \cite[Lemma 2.5]{PP}, \cite[Lemma 3]{CPP} 
and \cite[Lemma 2.6]{APR} to the functional setting $\W$.
\begin{lemma}
	\label{lemma2.5}
 Let \textbf{(H1)}-\textbf{(H2)} hold. Let $c \in L^q(\O)$ and $b \in \left( L^q(\O) \right)^{m_0}$
	for some $q >\frac{Q+2}{2}$ and let $f \in L^2(\O)$. Moreover, let us assume that $\div b$ in $\O$. Let $v$ be a non-negative weak sub-solution to $\L v = f$  
	in $\O$. For 
	every $\phi \in \D (\O)$, $\phi \ge 0$, and for almost every $z 
	\in \R^{N+1}$, we have 
	\begin{align*}
		\int_{\O} - \langle A Dv, D (\G (z, \cdot) \phi ) \rangle &+
		\G (z, \cdot) \phi Yv 
		+ \langle b, D v \rangle \G (z, \cdot) \phi  + c v \G(z, \cdot) \phi -  \G (z, \cdot) \phi f
		\ge 0.
	\end{align*}
	An analogous result holds for weak super-solutions to $\L u = f$.
\end{lemma}
\begin{proof}
For every $\varepsilon >0$, we set 
\begin{eqnarray}
\psi_\varepsilon(z,\z)=1-\chi_{\varepsilon,2\varepsilon}\left(\Vert \z^{-1} \circ z \Vert\right)
\end{eqnarray}
where $\chi_{\r,r} \in C^\infty([0,+\infty))$ is the cut-off function defined by
\begin{equation}\label{chi}
\chi_{\r,r}(s)=\left\{ \begin{array}{ll}
0,\quad &\textrm{if $s \geq r$},\\
1,\quad &\textrm{if $0\leq s \leq \r$},
\end{array} \right. \quad |\chi'_{\r,r}|\leq \frac{2}{r-\r},
\end{equation}
with $\frac{1}{2} \leq \r < r \leq 1$. As $v$ is a weak-subsolution, for every $\varepsilon >0$ and $z \in \R^{N+1}$, we have
\begin{align*}
	0 \leq	
-I_{1,\varepsilon}(z)+I_{2,\varepsilon}(z)-I_{3,\varepsilon}(z)+I_{4,\varepsilon}(z)+I_{5,\varepsilon}(z)+I_{6,\varepsilon}(z)		
	\end{align*}
where
\begin{equation*}
\begin{split}
I_{1,\varepsilon}(z)&=\int_\O [\langle A D v, D\G (z,\cdot) \rangle \phi \psi_\varepsilon(z,\cdot)](\z)d\z \qquad \qquad \qquad  
I_{3,\varepsilon}(z)=\int_\O [\langle A D v, D \psi_\varepsilon(z,\cdot)\rangle \phi \G(z,\cdot)](\z)d\z\\
I_{2,\varepsilon}(z)&=\int_\O [\G(z,\cdot
)\psi_\varepsilon(z,\cdot)(-\langle A D v, D \phi)\rangle + \phi Y v)](\z)d\z \quad
I_{4,\varepsilon}(z)=
\int_\O  \langle b, D v \rangle \G (z, \cdot) \phi \psi_\varepsilon(z,\cdot)](\z)d\z\\
I_{5,\varepsilon}(z)&=\int_\O [c v \G(z, \cdot) \phi \psi_\varepsilon(z,\cdot)](\z)d\z \qquad 
\qquad \quad \qquad \qquad 
I_{6,\varepsilon}(z)=-\int_\O [\G (z, \cdot) \phi \psi_\varepsilon(z,\cdot) f](\z)d\z
\end{split}
\end{equation*}
Keeping in mind Corollary \ref{corollary}, it is clear that the integrals that define $I_{i,\varepsilon}(z)$, $i=1,2,3$ are potentials and therefore convergent for almost every $z \in \R^{N+1}$. Thus, by a similar argument to the one used in \cite{PP} in the proof of Lemma 2.5, we infer that for almost every $z \in \R^{N+1}$
\begin{equation*}
\begin{split}
\lim_{\varepsilon \to 0^+}I_{1,\varepsilon}(z)&=\int_\O [\langle A D v, D\G (z,\cdot) \rangle \phi ](\z)d\z \qquad \qquad \lim_{\varepsilon \to 0^+}I_{3,\varepsilon}(z)=0\\
\lim_{\varepsilon \to 0^+}I_{2,\varepsilon}(z)&=\int_\O [\G(z,\cdot
)(-\langle A D v, D \phi)\rangle + \phi Y v)](\z)d\z ,
\end{split}
\end{equation*}
where the passage to the limit for the term $I_{2, \e}$ by a density argument.
We now take care of the term $I_{4,\varepsilon}$. Integrating by parts and taking advantage of the assumption $\div b \ge 0$, we obtain
\begin{align*}
I_{4,\varepsilon}(z)&=
-\int_\O [\div b \, \G(z,\cdot)\phi\psi_\varepsilon(z,\cdot)v](\zeta) d\zeta-\int_\O[\langle b,D\left(\G(z,\cdot)\phi\psi_\varepsilon(z,\cdot) \right) \rangle \,v](\zeta)d\zeta\\
&\leq -\int_\O[\langle b,D\left(\G(z,\cdot)\phi\psi_\varepsilon(z,\cdot) \right) \rangle \,v](\zeta)d\zeta
\end{align*}
We are left with the estimate of a potential and therefore we exploit once again Corollary \ref{corollary}. Since we have $b \in \left( L^q (\O) \right)^{m_0}$ and $v \in L^2(\O)$, we get 
\begin{equation*}
| \G(z,\cdot) | | \phi| |b| |D v| \in L^{2 \alpha}(\O),
\end{equation*}
where 
\begin{equation*}
\alpha=\alpha(q)=\frac{q(Q+2)}{q(Q-2)+2(Q+2)}>1 \qquad \text{if and only if} \qquad q > \frac{Q+2}{2}.
\end{equation*}
Hence, $|\langle b, D \left(\G(z,\cdot)  \phi \psi_\varepsilon(z,\cdot) \right) \rangle v|\leq | \langle b, D \left(\G(z,\cdot)   \phi \right) \rangle v   | \in L^{1}(\O)$.
Thus, the Lebesgue convergence theorem yields
\begin{align*}
\lim_{\varepsilon \to 0^+}I_{4,\varepsilon}(z)&=\lim_{\varepsilon \to 0^+}-\int_\O[\langle b,D\left(\G(z,\cdot)\phi\psi_\varepsilon(z,\cdot) \right) \rangle \,v](\zeta)d\zeta\\
&=\int_\O [\langle b, D \G (z, \cdot) \phi \rangle  v ](\z)d\z.
\end{align*}
Similarly, we can estimate the term $I_{5,\varepsilon}$ noting that $|c| |v| |\G(z,\cdot) | |\phi| \in L^{2 \alpha}(\O)$ 
with $\alpha$ as above. As a consequence, we have
\begin{equation*}
|c v \G(z,\cdot)  \phi \psi_\varepsilon(z,\cdot) | \leq |c v \G(z,\cdot)  \phi  | \in L^{1}(\O), \qquad \text{and thus} \, \, 
\lim_{\varepsilon \to 0^+}I_{5,\varepsilon}(z)=\int_\O [c v \G(z, \cdot) \phi ](\z)d\z.
\end{equation*}
Now, we are left with the estimate of term $I_{6,\varepsilon}$, which is again a $\G$-potential such that
\begin{equation*}
|\G(z,\cdot) | |\phi| |f| \in L^{2 \kappa}(\O),
\end{equation*}
where $\kappa=\frac{Q+2}{Q-2}$. Thus, we infer 
$|\G(z,\cdot)  \phi \psi_\varepsilon(z,\cdot) f| \leq |\G(z,\cdot)  \phi  f| \in L^{1}(\O)$.
Therefore we conclude the proof by applying the dominated convergence theorem to $I_{6,\varepsilon}(z)$.
\end{proof}

\medskip

We conclude this Section by recalling the following lemma for whose proof we refer to \cite[Lemma 6]{CPP}.
\begin{lemma}
	\label{cylinders}
	There exists a positive constant $\overline{c} \in (0,1)$ such that
	\begin{equation}
		\label{cylindereq}
		z \circ \Q_{\overline{c} r (r-\r)} \subseteq \Q_r,
	\end{equation}
	for every $0 < \r < r \le 1$ and  $z \in \Q_\r$. 
\end{lemma}

        \setcounter{equation}{0}\setcounter{theorem}{0}
\section{Local boundedness for weak solutions to $\L u = f$}
\label{moser}
Other than the new framework for the study of the weak regularity theory of \eqref{defL}, the local H\"older continuity for solutions and the Harnack inequality, the Moser's iterative scheme for weak solutions to \eqref{defL} with unbounded source term $f$, to whose proof this Section is devoted, is in itself one of the main novelties of our work. 

Indeed, this procedure was firstly introduced by Moser in \cite{M3, M4} and it is based on the iterative combination of a Caccioppoli and a Sobolev inequality. 
When considering the classical uniformly elliptic and parabolic settings, the Caccioppoli inequality provides us with a priori estimates for the $L^2$ norm of the complete gradient of the solution in terms of the $L^2$ norm of the solution and we are able to consider the classical Sobolev embedding to obtain a gain of integrability for the solution.

This is not the case when dealing with operator \eqref{defL}. In fact, the degeneracy of the diffusion part allows us to estimate only the partial gradient $D_{m_{0}} u$ of the solution to $\L u = f$ (see Theorem \ref{pcaccf}). In addition, according to our definition of weak solution, $u$ does not lie in a classical Sobolev space. 
In order to overcome these issues, we adopt a technique based on the representation of a solution $u$ to  $\L u = f$ in terms of the fundamental solution $\G$ (see \eqref{eq-Gamma0-b}) of the principal part operator $\K$. Indeed, following the idea presented for the first time in \cite{PP} and later on applied in \cite{APR, CPP}, we have that if $u$ is a 
solution to $\L u = f$ then
\begin{equation}
	\label{representation-formula}
	\K u = \left( \K - \L \right) u - f = \div_{m_{0}} \left( \left( \I_{m_{0}} - A_{0} \right) D_{m_{0}}
	u \right) - \langle B x , D u \rangle - \p_{t} u .
\end{equation}
Thus, by combining this representation formula with the potential estimates presented in Corollary \ref{corollary}, we are able to prove a Sobolev inequality (see Theorem \ref{sobolev}) and a Caccioppoli inequality (see Theorem \ref{pcaccf}) for weak solutions to $\L u = f$. We remark that in literature there is no similar result available for weak solutions in the sense of Definition \ref{weak-sol2} to the Kolmogorov equation $\L u = f$, with $f$ non zero source term. Indeed, in \cite{GIMV, GI, GM} the authors consider the case of the kinetic Fokker-Planck operator, whereas in \cite{PP, CPP, APR} the authors
consider the Kolmogorov equation $\L u = 0$ under the ``strong'' notion of weak solution, i.e. $Y u \in L^{2}$.

\begin{theorem}\label{boundedness}
	Let $u$ be a non-negative weak sub-solution to $\L u = f$ in $\O$ under the 
	assumptions \textbf{(H1)}-\textbf{(H3)}. Let $z_0 \in \O$ and $r, \r$, with $0 < \r 
	< r$, be such that $\overline{\Q_r(z_0)}\subseteq \O$. 
	Then for every $p \in \R$ such that $p > 1/2$ there exist a positive constant $C$ such that
	\begin{eqnarray*}
		\sup_{\Q_\r(z_0)}u^p \leq \frac{C }{(r-\r)^{4 \mu}} \parallel u^p \parallel_{L^{\b} (\Q_r)} 
	\end{eqnarray*}
	where $C= C \left( p,\l,\Lambda,Q,\parallel b \parallel_{L^q(\Q_r(z_0))},\parallel c \parallel_{L^q(\Q_r(z_0))},\parallel f 
	\parallel_{L^q(\Q_r(z_0))}\right)$ and 
	\begin{eqnarray} \label{est-it-2} 
		\mu := \frac{\a}{\a-\b}, \qquad \text{where } \quad \alpha:=\frac{q(Q+2)}{q(Q-2)+2(Q+2)} \quad \text{and} 
		\quad \beta:=\frac{q}{q-1}.
	\end{eqnarray}
\end{theorem}

\begin{remark} \label{remark-boundedness}
Theorem \ref{boundedness} holds true under the assumptions of Remark \ref{remarkWZ}. In particular, in this case the constant $\alpha$ is replaced by $\alpha=1+\frac{2}{Q}$, that was firstly obtained in \cite{PP}. This is due the fact that the Sobolev inequality, Theorem \ref{sobolev}, holds true with a greater exponent if we assume more integrability for the coefficient $b$. Thus, this allows us to obtain the local boundedness for weak solutions to $\L u=f$ with lower integrability for $c$ and $f$, i.e. $c, f \in L^{q}(\O)$, with $q > \frac{Q+2}{2}$.
\end{remark}
	
\subsection{Sobolev Inequality} This subsection is devoted to the proof of a Sobolev type inequality for weak solutions to $\L u = f$. Our approach is inspired by the paper \cite{PP} and 
allows us to construct an ``ad hoc'' Sobolev embedding for weak solutions to $\L u = f$ by overcoming the difficulties due to the degeneracy of the second order part of the operator $\L$ at the cost of lowering the Sobolev exponent, that in our case depends on $q$ and it is defined in \eqref{est-it-2}. We remark that the following statement holds true under lower integrability assumption than the one required in \textbf{(H3)}. 

\begin{theorem}[Sobolev Type Inequality for sub-solutions]  \label{sobolev}
	 Let \textbf{(H1)}-\textbf{(H2)} hold. Let $c \in L^q(\O)$ and $b \in \left( L^q (\O) \right)^{m_0}$
	for some $q >\frac{Q+2}{2}$ and let $f \in L^2(\O)$. 
	Let $v$ be a non-negative weak 
	sub-solution of $\L v = f$ in $\Q_1$. Then there exists a constant $C = C(Q,\l,\Lambda) > 0$ such that $v \in L^{2\a} (\Q_1)$, and the following inequality holds
	\begin{align*}
		\parallel v \parallel_{L^{2 \a}(\Q_\r(z_0))} \le &C \cdot \left(  
		\parallel b \parallel_{L^{q}(\Q_r(z_0))} + \frac{r - \r + 1}{r - \r} \right) \parallel D_{m_0} v  \parallel_{L^{2}(\Q_r(z_0))} + \\
		&\qquad \qquad + C \cdot \left( \parallel c \parallel_{L^{q}(\Q_r(z_0))} + \frac{\r + 1}{\r(r - \r)} \right)
		\parallel  v \parallel_{L^{2}(\Q_r(z_0))} +C \cdot \parallel  f \parallel_{L^{2}(\Q_r(z_0))}
	\end{align*}
	for every $\r, r$ with $\frac{1}{2} \le \r < r \le 1$ and for every $z_0 \in \O$, where $\a =\a(q)$ is defined as
	\eqref{est-it-2}. The same statement holds for non-negative super-solutions.
\end{theorem}

\begin{proof}
	Let $v$ be a non-negative weak sub-solution to $\L v = f$. We represent $v$ in terms of the fundamental 
	solution $\G$. To this end, we consider the cut-off function $\cc_{\r,r}$ defined in \eqref{chi}
	for $\frac{1}{2} \le \r < r \le 1$. Then, if we consider the test function
	\begin{equation}
		\label{cutoff}
		\psi (x,t) = \cc_{\r,r} ( \parallel (x,t) \parallel ),
	\end{equation}
the following estimates hold true
	\begin{equation}
		\label{ineq}
		|Y \psi| \le \frac{c_0}{\r (r-\r)},
		\hspace{8mm}
		|\p_{x_j} \psi| \le \frac{c_1}{r-\r}
		\hspace{2mm} \text{for } j = 1, \ldots, m_0
	\end{equation}
	where $c_0$, $c_1$ are dimensional constants.
	For every $z \in \Q_\r$, we have
	\begin{align}
		\label{representation}
		v(z) = v \psi(z) 
		&= \int_{\Q_r} \left[ \langle A_0 D(v \psi), D\G (z, \cdot) \rangle - \G (z, \cdot) Y(v \psi) \right] (\z) d(\z) \\ \nonumber
		&=  I_0(z) + I_1(z) + I_2(z) + I_3(z)
	\end{align}
	where
	\begin{align*}
		I_0 (z) &= 
		\int_{\Q_r} \left[ \langle b, D v
		\rangle \G (z, \cdot) \psi \right](\z)d\z \hspace{1mm} 
		\, + \,\int_{\Q_r} \left[ c v \G (z, \cdot) \psi \right](\z) d \z  -\int_{\Q_r}[\G(z,\cdot)\psi f](\z)d\z\\
		I_1 (z) &= \int_{\Q_r} \left[ \langle A_0 D\psi, D \G (z, \cdot ) \rangle 
		v \right](\z) d \z \hspace{1mm} - \hspace{1mm} \int_{\Q_r} \left[ 
		\G (z, \cdot) v Y\psi \right](\z) d\z = I_1^{'} + I_1^{''} , \\
		I_2 (z) &= \int_{\Q_r} \left[ \langle (A_0 - A) Dv, D \G (z, \cdot ) 
		\rangle \psi \right](\z) d \z \hspace{1mm} - \hspace{1mm}
		\int_{\Q_r} \left[ \G (z, \cdot ) \langle A Dv, D \psi \rangle 
		\right](\z) d \z \\
		I_3 (z) &= \int_{\Q_r} \left[ \langle A Dv, D( \G (z, \cdot ) \psi)
		\rangle \right](\z) d \z \hspace{1mm} - \hspace{1mm}
		\int_{\Q_r} \left[ \left( \G (z, \cdot ) \psi \right) Yv
		\right](\z) d \z \hspace{1mm}  \\
		&
		- \int_{\Q_r} \left[ \langle b, D v
		\rangle \G (z, \cdot) \psi \right](\z) d \z - \, \int_{\Q_r} \left[ c v \G (z, \cdot) \psi \right](\z) d \z +\int_{\Q_r}[\G(z,\cdot)\psi f](\z)d\z
	\end{align*}
	Since $v$ is a non-negative weak sub-solution to $\L v = f$, it follows from Lemma 
	\ref{lemma2.5} that $I_3 \le 0$, then
	\begin{equation*}
		0 \le v(z) \le I_0(z) + I_1(z) + I_2(z)  \hspace{4mm} \text{for a.e. } z \in \Q_\r.
	\end{equation*}
	To prove our claim is sufficient to estimate $v$ by a sum of $\G-$potentials.
	We start by estimating $I_0$. In order to do so, we recall that
	\begin{equation*}
		\langle b,D v \rangle, \, c v \,  \in L^{2 \frac{q}{q+2}} \hspace{4mm} 
		\text{for } b, c \in L^q, \hspace{2mm} q > \frac{Q+2}{2} \hspace{2mm}
		\text{and } v \in L^2.
	\end{equation*}
	Thus by Corollary \ref{corollary} we get
	\begin{equation*}
		\G * \langle b, Dv \rangle , \G * (c v ) \in L^{2\a},
	\end{equation*}
	where $\a = \a(q)$ is defined in \eqref{est-it-2}. In addition, for $f \in L^2$, we have
	\begin{equation*}
	\G * f \in L^{2\kappa}, \quad \kappa=\frac{Q+2}{Q-2}.
	\end{equation*}
Observing that $\kappa > \alpha$, we obtain that $\G * f \in L^{2\alpha}$ and therefore
	\begin{align*}
	\parallel I_0 (\z) \parallel_{L^{2\a}(\Q_\r)} 	&\leq 
		\G * \left( \langle b,  D_{m_0}v \rangle \psi \right)  + \G * \left( c v \psi \right) 
		\parallel_{L^{2 \alpha}(\Q_\r)}+\parallel \G * f  \parallel_{L^{2 \alpha}(\Q_\r)}\\
		&\le C \cdot \left( 
		\parallel b \parallel_{L^q (\Q_\r)}  \parallel D_{m_0}v \parallel_{L^2 (\Q_\r)}
		+  \parallel c \parallel_{L^q (\Q_\r)} \parallel v \parallel_{L^2 (\Q_\r)}+  \parallel f \parallel_{L^2 (\Q_\r)}\right).
	\end{align*}
	We now deal with the $I_1$. $I_1'$ can be 
	estimated by \eqref{stima2} of Corollary \ref{corollary} as follows
	\begin{equation*}
		\parallel I_1 ' \parallel_{L^{2\a}(\Q_\r)} \le C
		\parallel I_1 ' \parallel_{L^{2^*}(\Q_\r)} \le C \parallel v 
		D_{m_0} \psi \parallel_{L^2 (\R^{N+1})} \le \frac{C}{r - \r} 
		\parallel v \parallel_{L^2(\Q_\r)},
	\end{equation*}
	where the last inequality follows from \eqref{ineq}. To estimate $I_1 ''$
	we use \eqref{stima1}
	\begin{align*}
		\parallel I_1 '' \parallel_{L^{2\a}(\Q_\r)} &\le C
		\parallel I_1 '' \parallel_{L^{2^*}(\Q_\r)} \le 
		\text{meas}(\Q_\r)^{2/Q}
		\parallel I_1 '' \parallel_{L^{2^{**}}(\Q_\r)} \\ 
		&\le 
		C \parallel v Y \psi \parallel_{L^2(\R^{N+1})} \le \frac{C}{\r (r-\r)} 
		\parallel v \parallel_{L^2 (\Q_\r)}.
	\end{align*}
	We can use the same technique to prove that 
	\begin{equation*}
		\parallel I_2 \parallel_{L^{2\a}(\Q_\r)} \le C \left( 1 + \frac{1}{r - \r} \right) \parallel
		D v \parallel_{L^2(\Q_\r)},
	\end{equation*}
	for some constant $C=C(Q, \l,\Lambda)$.  
	A similar argument proves the thesis when $v$ is a super-solution to $\L v = f$. In this 
	case we introduce the following auxiliary operator
	\begin{equation}
		\K = \div (A_0 \, D ) + \widetilde{Y}, \qquad \qquad 
		\widetilde{Y} \equiv - \langle x, B D \rangle - \p_t \, .
	\end{equation}
	Then we proceed analogously as in \cite{PP}, Section 3, proof of Theorem 3.3.
\end{proof}

\subsection{Caccioppoli inequality}
The aim of this subsection is to prove a Caccioppoli type inequality for powers $v = u^p$ of weak sub-solutions to $\L u = f$.
\begin{theorem}[Caccioppoli type inequality for sub-solutions]
	\label{pcaccf} Let \textbf{(H1)}-\textbf{(H3)} hold. Let $r, \r$ be such that $\frac{1}{2} \le \r < r \le 1$. 
	Then for every weak sub-solutions to $\L u = f$ we have that
	for every $p > 1/2$ it holds
			\begin{align} \label{caccf2} 
				\parallel D_{m_0} v \parallel^2_{L^2(\Q_\r)}
				\le \frac{4p}{\l(2p-1)} &\Big ( \frac{2p}{2p-1}  \frac{c_1^2 \Lambda}{(r-\r)^2} \, + \, \frac{c_0}{\r(r-\r)} \, 
						 + \frac{ \parallel b 
							\parallel^2_{L^q(\Q_r)} }{2\e_b}    \\ \nonumber
				&\qquad \qquad \qquad  
				 +  p \parallel c \parallel_{L^{q}(\Q_r)}   + 
					p \parallel f \parallel_{L^{q}(\Q_r)} \Big )
				\parallel u^p \parallel_{L^{2\b}(\Q_r)}^2
		\end{align}
	where $\b = \b(q) = \frac{q}{q-1}$, $\e_b = \frac{|2p-1|\l}{2 |p|}$ and $c_0, c_1$ are defined in \eqref{ineq}.
\end{theorem}
\begin{remark}
	If $f \in L^2(\O)$ then the estimate \eqref{caccf2} holds true
	for every $\frac12 < p \le 1$.
\end{remark}
For the sake of completeness and in order to simplify the exposition of the proof of Theorem \ref{pcaccf} we consider the following lemma, which is the analogous to \cite[Proposition 3.2]{APR} when considering Definition \ref{weak-sol2} of weak solution. 
\begin{lemma}
	\label{pcaccB1} Let \textbf{(H1)}-\textbf{(H3)} hold. Let $p \in \R$, $p \ne 0$, $p \ne 1/2$ 
	and let $r, \r$ be such that $\frac{1}{2} \le \r < r \le 1$. Then for every weak sub-solution to $\L u = 0$ the following estimate holds true 
	\begin{align} \label{cacc1}
		\l \Big | \frac{2p-1}{p} \Big |  \, \parallel D_{m_0} u^p \parallel^2_{L^2(\Q_\r)}
		&\le  \Big ( \Big | \frac{2p}{2p-1} \Big | \frac{c_1^2 \lambda}{(r-\r)^2} \, + \, \frac{c_0}{\r(r-\r)} \, + 
		\\ \nonumber
		 &\qquad \qquad \qquad \quad \,
		  + \frac{c_1}{r-\r}\parallel b \parallel^2_{L^q(\Q_r)}  + |p| \parallel c \parallel_{L^{q}(\Q_r)} \Big )
		\parallel u^p \parallel_{L^{2\b}(\Q_r)}^2 , \nonumber
	\end{align}
	where $\b = \b(q) = \frac{q}{q-1}$ and $c_0, c_1$ are defined in \eqref{pcaccf}. 
\end{lemma}
\begin{proof}
	Let us consider $p<1$, $p \ne 0$, $p \ne 1/2$. First of all, we consider a 
	non-negative weak sub-solution $u$ to $\L u = 0$. For every $\psi \in C^{\infty}_0(\Q_r)$ 
	we consider the function $\phi = u^{2p-1}\psi^2$. Note that $\phi,
	D_{m_0}\phi \in L^2 (\Q_r)$, then we can use $\phi$ as a test function in the weak formulation
	\eqref{kolmo}:
	\begin{align*}
	 0 &\le  \int_{\Q_r} - \langle A Du, D(u^{2p-1}\psi^2) \rangle + u^{2p-1}\psi^2 Y u 
	 + \langle b , Du \rangle u^{2p-1}\psi^2 + c u^{2p}\psi^2 
	\end{align*}
	Let $v = u^p$. Since $u$ is a weak sub-solution to $\L u = 0$  and $u 
	\ge u_0$, thanks to Proposition \ref{sobolev} we have that $v, D_{m_0} v \in L^2(\Q_r)$ and $Y v \in L^2H^{-1}(\Q_r)$ and thus:
	\begin{align} \label{e-cacc}
		0 \le &- \frac{2p-1}{p} \int_{\Q_r}  \langle A Dv , Dv \rangle  \psi^2 \, - 2 \int_{\Q_r} \langle A Dv, D \psi \rangle v \psi 
		\,  + \int_{\Q_r}  v \, Y v \, \psi^2 \, +  \\ \nonumber
		&
		+ \, \int_{\Q_r}  v \,  \langle b, D_{m_0} v \rangle \, \psi^2
		\, + p \, \int_{\Q_r} c v^2 \psi^2.
	\end{align}
	First of all, we estimate of the terms involving the matrix $A$. In particular, by applying Young's inequality we have that for every $\e > 0$:
		\begin{align}
			\label{youngA}
			2 \int_{\Q_r} |\langle A Dv, D \psi \rangle v \psi|  \le 
			\e \int_{\Q_r} \langle A Dv, Dv \rangle \psi^2 \, + \, 
			\frac{1}{\e}\int_{\Q_r} \langle A D \psi, D \psi \rangle v^2.
		\end{align} 
	Thus, 
	taking the absolute value on the right-hand side of \eqref{e-cacc} and considering $\e = |
	(2p-1)/2p| $ in \eqref{youngA}, for every $p < 1$, $p \ne 0, 1/2$ we obtain: 
	\begin{align} \label{e-cacc2}
	 \Big | \frac{2p-1}{p} \Big | \l \int_{\Q_{\r}}|D_{m_0} v|^2 \, 
		&\le \Big | \frac{2p \lambda}{2p-1} \Big | \frac{c_1^2}{(r-\r)^2}\int_{\Q_r} v^2 \,  
		+ \, \frac{c_0}{\r(r-\r)} \int_{\Q_{r}}  v^2 \, + \\ \nonumber
		&+ \,\boxed{ \int_{\Q_r}  v \,  \langle b, D_{m_0} v \rangle \, \psi^2 }_B \,  + 
		\, \boxed{  |p| \int_{\Q_r} |c| v^2 \psi^2  }_C,
	\end{align}
	where we have considered the definition \eqref{cutoff} of the cut-off function $\psi$ in order to estimate the first integral on the 
	right-hand side.
	By considering the bounds on \eqref{ineq}, assumption \textbf{(H3)} and the H\"older's inequality with exponent $\b = \frac{q}{q-1}$ we are in position to estimate the term 
	B
	\begin{align*}
		\boxed{ \frac12 \int_{\Q_r} \langle b, D_{m_0} v^2 \rangle \, \psi^2 }_B 
		&= - \frac12 \int_{\Q_r} \div \cdot b\,  v^2 \, \psi^2 -
		\int_{\Q_r} \langle b,  \psi D_{m_0} \psi \rangle v^2 \\
		&\le \, \frac{C}{r -\r} \parallel b \parallel_{L^q(\Q_r)} \, \parallel v \parallel^2_{L^{2\b}(\Q_r)} 
	\end{align*}
	and the term C:
	\begin{equation*}
		\boxed{ |p| \int_{\Q_r} |c| v^2 \psi^2  }_C \le |p| \, \parallel c \parallel_{L^{q}(\Q_r)} \, \parallel v \parallel_{L^{2\b}(\Q_r)}^2.
	\end{equation*}
	Thus, by by choosing $\e_a = \lambda / |p|$ and $\e_b = (|2p-1|\l)/(2 |p|)$
	we recover \eqref{cacc1}.

	Now, let us consider the case $p \ge 1$. For any $n \in \N$, we define the function 
	$g_{n,p}$ on $]0, +\infty [$ as follows 
	\begin{equation*}
		g_{n,p}(s) = \begin{cases}
			s^p, &\text{if} \hspace{2mm} 0 < s \le n, \\
			n^p + p n^{p-1} (s-n), \hspace{4mm} &\text{if} \hspace{2mm} s 
			> n,
		\end{cases}
	\end{equation*}
	then we let $v_{n,p} = g_{n,p}(u)$.
	Note that $g_{n,p} \in C^1( \R^+)$ and $g_{n,p}^{'} \in L^\infty( \R^{+})$.
	Thus, since $u$ is a weak sub-solution to $\L u = 0$, we have
	\begin{equation*}
		v_{n,p} \in L^2, \hspace{2mm} D_{m_0} v_{n,p} \in L^2, \hspace{2mm} 
		Y v_{n,p} \in L^2H^{-1}.
	\end{equation*}
	We also note that the function 
	\begin{equation*}
		g^{''}_{n,p} (s) = \begin{cases}
			p (p-1) s^{p-2}, \hspace{4mm} & \text{if} \hspace{2mm} 0 < s < n \\
			0, & \text{if} \hspace{2mm} s \ge n,
		\end{cases}
	\end{equation*} 
	is the weak derivative of $g_{n,p}^{'}$, then $D_{m_0} g_{n,p}^{'}(u) =  g^{''}_{n,p} (u) D_{m_0} (u)$
	(for the detailed proof of this assertion, we refer to \cite{GT}, Theorem 7.8). Hence, by
	considering
	\begin{equation}\label{testfunction}
		\phi = g_{n,p}(u) \hspace{1mm} g_{n,p}^{'}(u) \hspace{1mm} \psi^2, 
		\hspace{4mm} \psi \in C^\infty_0(\Q_r)
	\end{equation}
	as a test function in the weak formulation
	\eqref{kolmo} and recalling that we find 
	$v_{n,p}=g_{n,p} (u)$ we have
	\begin{align}\label{eq-vnp} 
		0  &\le \int \limits_{\Q_r} - \psi^2 \langle A Dv_{n,p}, D v_{n,p} \rangle - \int \limits_{\Q_r} \frac{v^{''}_{n,p}(u) \, v_{n,p} (u)}{\left( v^{'}_{n,p} (u) \right)^2} \psi^2 \langle A Dv_{n,p}, D v_{n,p} \rangle + \\  \nonumber
		&\qquad - \int \limits_{\Q_r} 2 \psi \langle A Dv_{n,p}, D \psi \rangle v_{n,p}  
		+ \int \limits_{\Q_r} \psi^2 \,v_{n,p} Y(v_{n,p})  + \\ \nonumber
		&\qquad+ 	\boxed{\int \limits_{\Q_r} \langle b, D_{m_0} v_{n,p} \rangle v_{n,p} \, \psi^2}_B \nonumber
		+ \boxed{\int \limits_{\Q_r} c u v_{n,p} (u) \, v^{'}_{n,p} (u) \psi^2}_C. 
\end{align}	
Let us firstly estimate the term $B$ as in the previous case:
\begin{align*}
\boxed{\frac12 \int \limits_{\Q_r} \langle b, D_{m_0} v^2_{n,p} \rangle \, \psi^2}_B 
\le \,
\int_{\Q_r} |\langle b,  \psi D_{m_0} \psi \rangle| v_{n,p}^2 \, 
\le \, \frac{C}{r -\r} \parallel b \parallel_{L^q(\Q_r)} \, \parallel v_{n,p} \parallel^2_{L^{2\b}(\Q_r)}
\end{align*}
Then the term $C$ is treated as follows:
\begin{equation*}
\begin{split}
\boxed{\int \limits_{\Q_r} c u v_{n,p} (u) \, v^{'}_{n,p} (u) \psi^2}_C \leq p \int \limits_{\Q_r} |c| \, |u^{2p}| \, \psi^2 \leq p \parallel c\parallel_{L^q(\Q_r)}\parallel u^p \parallel^2_{L^{2\b_2}(\Q_r)}.
\end{split}
\end{equation*}
Thus
we rewrite equation \eqref{eq-vnp} as follows
\begin{align}\label{est-Advnp}
&\int\limits_{\Q_r}\psi^2 \langle A D v_{n,p},D v_{n,p} \rangle \left(1+\frac{v^{''}_{n,p}(u) \, v_{n,p} (u)}{\left( p u^{p-1} \right)^2}-\varepsilon\right)\\
&\qquad \leq \left(  \frac{c_1^2\Lambda}{\e(r-\r)^2}
+\frac{c_0}{\r(r-\r)} 
+\parallel b \parallel_{L^q(\Q_r)} \, 
+p \parallel c\parallel_{L^q(\Q_r)} \right) \parallel u^p \parallel^2_{L^{2\b_2}(\Q_r)}.
\end{align}
We now observe that
\begin{equation}
\left\vert \psi^2 \langle A Dv_{n,p},Dv_{n,p} \rangle \frac{v^{''}_{n,p}(u) \, v_{n,p} (u)}{p^2u^{2p-2}} \right\vert \leq \frac{p-1}{p}\langle A Dv_{n,p},D v_{n,p} \rangle \in L^1(\Q_r)
\end{equation}
and thus, keeping in mind that $| D v_{n,p}  | \, \uparrow \, | D u^p | $ as $ n \rightarrow \infty$, we can apply the dominated convergence theorem to \eqref{est-Advnp}. The proof is complete if we choose $\e$ as in the previous case.
\end{proof}

\medskip

\textit{Proof of Theorem \ref{pcaccf}.}
We are now in a position to consider weak sub-solutions to $\L u = f$, with $f$ a non zero source term under assumption \textbf{(H3)}. First of all, let us consider the case where $1/2 < p \le 1$. Following the lines of the first part of this proof, 
for every $\psi \in C^{\infty}_0(\Q_r)$ we consider the function $\phi = u^{2p-1}\psi^2$ as a test function in the weak formulation of \eqref{kolmo} with non zero source term. Indeed,
by denoting $v = u^p$ we have that also in this case $v, D_{m_0} v \in L^2(\Q_r)$ and $Y v \in L^2H^{-1}$. 
Thus, we get exactly estimate \eqref{e-cacc2} with the following additional term regarding the source term $f$:
\begin{equation*}
	 \boxed{  |p| \int_{\Q_r} |f| u^{2p-1} \psi^2 }_F
\end{equation*}
In order to conclude the proof, we therefore need to properly estimate $u^{2p-1}$. In particular, if
\begin{itemize}
	\item $0 \le u < 1$ we have that $u^{2p-1} < 1$, thus
			\begin{equation} \label{F1}
				\boxed{ | p| \int_{\Q_r} |f| u^{2p-1} \psi^2 }_F \le 
				|p| \int_{\Q_r} |f| \psi^2 
			\end{equation}
	\item $u \ge 1$ we have that $u^{2p-1} < u^p = v$. By combining H\"older's inequality and
			Young's inequality \eqref{youngA} for every $\e_f > 0$ we get
			\begin{equation} \label{F2}
				\boxed{ |p| \int_{\Q_r} |f| u^{2p-1} \psi^2 }_F \le 
				|p| \| f \|_{L^2(\Q_r)} \| v \|_{L^2(\Q_r)} 
			\end{equation}
\end{itemize}
Thus, by combining \eqref{F1} with \eqref{F2} and considering $\e_f = 1$ we get \eqref{caccf2}.

We now address the case $p>1$. Reasoning as above, we set $v_{n,p} = g_{n,p}(u)$ and we choose the test function as in \eqref{testfunction}. Then the weak formulation of \eqref{kolmo} for $v_{n,p}$ reads exactly as \eqref{eq-vnp}
with the additional term 
\begin{equation*}
	\boxed{-p\int \limits_{\Q_r} f v_{n,p} (u) \, v^{'}_{n,p} (u) \psi^2}_F.
\end{equation*}
We deal with the boxed term $F$ by distinguishing two different cases. In particular, if 
\begin{itemize}
	\item $0 \le u < 1$ we have that $u^{2p-1} < 1$, thus we exactly recover estimate \eqref{F1}.
	\item $u \geq 1$ we have that $u^{2p-1} < u^{2p} $. We now observe that $u^p$ is a non-negative weak super-solution to $\L u =f$ and therefore $u^{p} \in L^{2\beta}$ in virtue of Theorem \ref{sobolev}. By applying H\"older's inequality, we then infer
		\begin{equation} \label{F2.2}
			\boxed{-p\int \limits_{\Q_r} f v_{n,p} (u) \, v^{'}_{n,p} (u) \psi^2}_F\le 
			 p \int_{\Q_r} |f| u^{2p-1} \psi^2 \le 
			p \| f \|_{L^q(\Q_r)} \| u^p \|^2_{L^{2 \b}(\Q_r)}.
		\end{equation}
\end{itemize}
By combining \eqref{F1} and \eqref{F2.2} we conclude the proof \eqref{caccf2}.
$\hfill \square$

\subsection{Proof of Theorem \ref{boundedness}}
Keeping in mind Lemma \ref{cylinders}, it suffices to give a proof in the case $z_0=0$ and to consider $\frac12 \le \r < r \le 1$,
since the general statement is obtained by applying to \eqref{moserp} the group dilations \eqref{gdil} and translations \eqref{grouplaw}. Combining Theorems
\ref{sobolev} and \ref{pcaccB1} for a non-negative sub-solution $u$, we obtain the following estimate. If $s>1/2$, $\d >0$ 
verify the condition 
\begin{equation*}
	|s - 1/2| \ge \d,
\end{equation*}
then, for every $\r, r$ such that $\frac{1}{2} \le \r < r \le 1$ we have  
\begin{align}
	\label {it}
	\parallel u^s \parallel_{L^{2\a} (\Q_\r)} \le \widetilde{C} \left( s, \l, \Lambda, Q, 
	\parallel b \parallel_{L^{q}(\Q_{r})} , \parallel c \parallel_{L^{q}(\Q_{r})},\parallel f \parallel_{L^{q}(\Q_{r})} \right)  
	\parallel u^s \parallel_{L^{2\b} (\Q_r)}
\end{align}
for a positive constant $\widetilde{C}$ that we be estimated as follows
\begin{align} \label{ctilde}
	\widetilde{C} ( s, \l,\Lambda,Q 
	, \parallel b \parallel_{L^{q}(\Q_{r})} &, \parallel c \parallel_{L^{q}(\Q_{r})},\parallel f \parallel_{L^{q}(\Q_{r})} ) \\ \nonumber
	&\le \frac{K( s, \l,\Lambda,Q, 
		\parallel b \parallel_{L^{q}(\Q_{r})} , \parallel c \parallel_{L^{q}(\Q_{r})},\parallel f \parallel_{L^{q}(\Q_{r})} ) }{\left( r - \r\right)^{2}},
\end{align}
since for every $\r, r$ of our choice the following estimates hold
\begin{equation*}
\begin{split}
\frac{1}{\left(\r(r-\r)\right)^{\frac{1}{2}}(r-\r)}\leq \frac{1}{\left(r-\r\right)^{2}},\qquad
\frac{1}{\r(r-\r)}\leq \frac{1}{\left(r-\r\right)^{2}}, \qquad \frac{1}{\left(\r(r-\r)\right)^{2}}\leq \frac{1}{(r-\r)^{2}}.
\end{split}
\end{equation*}
Then we set $v=u^{\frac{p}{2}}$.
Fixed a suitable $\d >0$, that we will specify later on, and a suitable $p \geq1$ such that 
\begin{equation}
	\label{pcond}
	p \, \left( \frac{\a}{\b} \right)^n - 1 \ge 2 \d , \qquad \forall n \in \N \cup \{ 0 \},
\end{equation}
 we iterate inequality \eqref{it} by choosing
\begin{equation*} 
	\r_n = \r + \frac{1}{2^{n}} \left( r - \r \right), \qquad 
	p_n = \left( \frac{\a}{\b} \right)^n \, \frac{p}{2}, \hspace{6mm} n \in \N \cup \{ 0 \}. 
\end{equation*} 
Thus, by combining \eqref{it} and \eqref{ctilde}, for every $n \in \N \cup \{ 0 \}$ the following holds  
\begin{align*}
	\parallel v^{( \frac{\a}{\b})^{n}} &\parallel_{L^{2\a} (\Q_{\r_{n+1})}} 
	\le \frac{ K( p, \l,\Lambda,Q, 
		\parallel b \parallel_{L^{q}(\Q_{r})} , \parallel c \parallel_{L^{q}(\Q_{r})},\parallel f \parallel_{L^{q}(\Q_{r})} ) }{\left( \r_n - \r_{n+1}\right)^{2}} 
	\parallel v^{( \frac{\a}{\b})^{n}} \parallel_{L^{2\b} (\Q_{\r_n})}.
\end{align*}
From now on we denote $K =  K( p, \l,\Lambda,Q, 
\parallel b \parallel_{L^{q}(\Q_{r})} , \parallel c \parallel_{L^{q}(\Q_{r})},\parallel f \parallel_{L^{q}(\Q_{r})} )$. Since 
\begin{equation*}
	\parallel v^{\left( \frac{\a}{\b}\right)^n} \parallel_{L^{2\a}} = \left( \parallel v \parallel_{L^{2\a \left( \frac{\a}{\b} \right)^n}} \right)^{\left( \frac{\a}{\b}\right)^n} \qquad {\rm and } \qquad 
	\parallel v^{\left( \frac{\a}{\b} \right)^n} \parallel_{L^{2\b}} = \left( \parallel v \parallel_{L^{2	\b \left( \frac{\a}{\b} \right)^{n}}} \right)^{\left( \frac {\a}{\b}\right)^{n}}
\end{equation*}
we are able to rewrite the previous estimate in the following form for every $n \in \N \cup \{ 0 \}$
\begin{align*}
	 \parallel v \parallel_{L^{2\a \left( \frac{\a}{\b} \right)^n} \left( \Q_{\r_{n+1}} \right)}
	 \le \left( 
	\frac{ K }{\left( \r_n - \r_{n+1}\right)^{2}}  
	\right)^{\left( \frac{\b}{ \a} \right)^n} \parallel v \parallel_{L^{2	\b \left( \frac{\a}{\b} \right)^{n}} (\Q_{\r_n})}.
\end{align*}
Iterating this inequality
and letting $n$ go to infinity, we get
\begin{equation*}
	\sup_{\Q_{\r}} v \, \le \, \frac{ \widetilde{K} }{ (r -\r)^{2\m} } \parallel v \parallel_{L^{2\b} (\Q_r)},
	\qquad \text{where } \, \m = \frac{\a}{\a- \b} \quad \text{and}
	\quad \widetilde{K} = \prod_{j=0}^{+ \infty}  
	\left[ K 2^{2(j+2)}
	\right]^{\left( \frac{\b}{\a} \right)^j}.
\end{equation*}
We remark that $\widetilde{K}$ is a finite constant depending on $\d$, since the product over $j$ 
corresponds to a converging series. 
Thus, for every $p\geq 1$ which verifies condition \eqref{pcond} we have proved that 
\begin{equation}
	\label{moserp}
	\sup_{\Q_{\r}} u^p \, \le \, \frac{ \widetilde{K}^2 }{ (r -\r)^{4 \mu } } \parallel u^p \parallel_{L^{\b} (\Q_r)},
\end{equation}
 and thus the statement is proved. We now make a suitable choice of $\d > 0$, only dependent on the homogeneous dimension $Q$, and on $q$ in order to show that \eqref{pcond} holds for every $p\geq 1$. We notice that, if $p$ is a number of the form
\begin{equation*}
	p_m = \frac12 \left(  \frac{\a}{\b} \right)^{m} \left( \frac{\a}{\b} + 1 \right)  , \qquad m \in \Z,
\end{equation*}
then \eqref{pcond} is satisfied with the following choice of $\d$ for every $m \in \Z$
\begin{equation*}
	\d = \frac{\a - \b}{8\b} .
\end{equation*}
Therefore \eqref{moserp} holds for such a choice of $p$, with $\widetilde{K}$ only dependent on $Q, q, \l$, $\Lambda$ and $\parallel a \parallel_{L^q(\Q_r)}$, $\parallel b \parallel_{L^q(\Q_r)}$, $\parallel c \parallel_{L^q(\Q_r)}$. On the other hand, if $p$ is an arbitrary 
positive number, we consider $m \in \Z$ such that
\begin{equation*}
	p_m \le p < p_{m+1}.
\end{equation*}
Hence, the proof is complete.
\hfill $\square$

\medskip

\setcounter{equation}{0}\setcounter{theorem}{0}
\section{Weak Poincaré inequality}
\label{poincare}
This section is devoted to the proof of a weak Poincaré inequality (see Theorem \ref{weak-poincare}) that holds true for every function $u \in \W$. As one immediately understands, this Poincaré inequality is independent of the equation $\L u = f$ and only relies on the structure of the space $\W$. Its importance lies in the fact that it is a necessary tool in the proof of a Harnack inequality (see Theorem \ref{harnack-thm}) and of the local H\"older continuity (see Theorem \ref{local-holder}) of a solution $u$ to \eqref{defL}.
In order to state our result, we first need to introduce the following sets
\begin{align}\label{Qzero}
	\Q_{zero}&=\lbrace (x,t) :  |x_j|\leq \eta^{\alpha_j}, j=1,\ldots,N, -1-\eta^2 < t \leq -1\rbrace,\\ \nonumber
	\Q_{ext}&=\lbrace (x,t) :  |x_j|\leq 2^{\alpha_j}R, j=1,\ldots,N, -1-\eta^2 < t \leq 0\rbrace,
\end{align}
where $R>1$, $\eta \in (0,1]$ and the exponents $\alpha_j$, for $j=1,\ldots,N$, are defined in \eqref{alphaj}. 
Indeed, $\Q_{zero}$ and 
$\Q_{ext}$ are completely equivalent to \eqref{rcylind} thanks to Remark \ref{ball-cil}, but they are more convenient for the construction of the cut-off function 
$\psi_{1}$ introduced in Lemma \ref{coc}. 
Now, we state our weak Poincaré inequality.

\begin{theorem}[Weak Poincaré inequality]\label{weak-poincare}
	Let $\eta \in (0,1)$ and let $\Q_{zero}$ and $\Q_{ext}$ be defined as in \eqref{Qzero}. Then there exist $R>1$ and $\vartheta_0 \in (0,1)$ such that for any non-negative function $u \in \W$ such that $u \leq M$ in $\Q_1$ for a positive constant $M$,  we have
	\begin{equation}
		\Vert (u-\vartheta_0 M)_+ \Vert_{L^2(\Q_1)}\leq C\left(\Vert D_{m_0} u \Vert_{L^2(\Q_{ext})}+\Vert Y u \Vert_{L^2H^{-1}(\Q_{ext})} \right),
	\end{equation} 
	where $C>0$ is a constant only depending on $Q$.
\end{theorem}
\noindent
The notation we consider here needs to be intended in the sense of \eqref{normW}. 
In particular, we have that $L^2H^{-1}(\Q_{ext})$ is short for 
\begin{equation}
	\label{notazioneq}
	L^2(B_{2^3 R}\times \ldots\times B_{2^{2\kappa+1} R}\times(-1-\eta^2,0],H^{-1}_{x^{(0)}}(B_{2 R})),
\end{equation}
 where we have split $x=\big(x^{(0)},x^{(1)},\ldots, x^{(\kappa)}\big)$ according to \eqref{split.coord.RN}.

\medskip 

We observe that a somewhat similar Poincar\`{e} inequality was already introduced by W. Wang and L. Zhang for subsolutions of ultraparabolic equations, see for instance \cite[Lemmas 3.3 \& 3.4]{WZ4}  and the
corresponding lemmas in \cite{WZ3}, for the ``strong'' notion of weak solution introduced in \cite{PP}, i.e. $Yu \in L^{2}$.  However, the statement of our Theorem \ref{weak-poincare} differs  substantially from the ones in the articles mentioned above. This is mainly due to the fact that in our treatment we extend the functional framework introduced in \cite{AM} to study the Kolmogorov-Fokker-Planck operator. As already mentioned, this framework seems to be the appropriate one when dealing with operator $\L$ and allows us to forget the equation under study. Moreover, the proof we propose here is different from the ones in \cite{WZ4,WZ3}, as we avoid using repeatedly the exact form of
the fundamental solution of $\L$ and exploit arguments closer to
the classical theory of parabolic equations (developed, for instance, in \cite{lieberman}). 
Additionally, the information obtained through the log-transform is here summarized in
only one weak Poincaré inequality (this is in contrast with \cite{WZ4, WZ3}, where it is split in several lemmas). Finally, we remark that in this paper
the geometric settings of the main lemmas are simpler than the ones in \cite{WZ4}, \cite{WZ3}. 

\medskip

In order to prove Theorem \ref{weak-poincare}, the idea is to firstly derive a local Poincaré inequality in terms of an error function $h$, that is defined as the solution to the following Cauchy problem
\begin{equation}\label{defpsi}
	\left\{ 
	\begin{array}{ll} 
		\K h=u \K \psi,\quad &\textit{in $  \mathbb{R}^{N} \times (-\r^2,0)$}\\
		h=0,\quad &\textit{in $  \mathbb{R}^{N}\times \lbrace -\r^2 \rbrace$}
	\end{array} \right.
\end{equation}
where $\K $ is the operator defined in \eqref{defK} and $\psi$ is a given cut-off function, and then to explicitly control the error $h$ 
through the $L^{\infty}$ norm of the function $u$ (see Lemma \ref{localizationlem}). 

\begin{lemma}
	\label{weak-poincare-h}
	Let $\Q_{ext}$ be as defined in \eqref{Qzero} and let $\psi :	
	\mathbb{R}^{N+1} \rightarrow [0,1]$ be a $C^\infty$ function, with support in $\Q_{ext}$ and such that $\psi=1$ in $\Q_1$. Then for any $u \in 
	\W$, the following holds 
	
	\begin{equation}\label{weakpoincareest}
		\Vert (u-h)_+ \Vert_{L^2(\Q_1)}\leq C \left( \Vert D_{m_0} u \Vert_{L^2(\Q_{ext})} + \Vert Y u \Vert_{L^2H^{-1} (\Q_{ext}) }
		\right)
	\end{equation}
	where $h$ is defined in \eqref{defpsi}, $C$ is a constant only depending on $|\r^2|$ and $\Vert D_{m_0}\psi 	
	\Vert_{L^\infty(\Q_{ext})}$, and the notation we consider needs to be intended in the sense of \eqref{notazioneq}.
\end{lemma}
\noindent
This 
local weak Poincaré inequality is an extension to operator $\L$ and to the space $\W$ of the one proved in \cite{GI} and a simplification of 
the one proved in \cite{WZ4, WZ3}. Moreover, the result holds true for any cylinder of the form $\Q_{ext}=B_{R_0}\times\ldots\times B_{R_\kappa}\times(-\r^2,0]$, provided that $\overline{\Q_1}\subset \Q_{ext}$. The proof of Lemma \ref{weak-poincare-h} is
mainly based on the properties of the principal part operator $\K$ and of $\W$.

\medskip

\textit{Proof of Theorem \ref{weak-poincare-h}.}
In virtue of Remark \ref{remark-Y}, the function $g:=u\psi$ satisfies the following equation in the sense of distributions
\begin{equation*}
\K g = u \K \psi + \div_{m_{0}} \tilde{H}_1+\tilde{H}_0, \quad \textrm{in $\mathcal{D}'(\mathbb{R}^N \times (-\r^2,0))$},
\end{equation*}
where $\tilde{H}_1=(H_1+D_{m_0} u)\psi$ and $\tilde{H}_0=H_0\psi -H_1 D_{m_0} \psi + \langle D_{m_0} \psi, D_{m_0} u \rangle$. Thus, owing to \eqref{defpsi}, we obtain
\begin{equation}\label{Kg-h}
\K (g-h) =  \div_{m_{0}} \tilde{H}_1+\tilde{H}_0=:\tilde{H}, \quad \textrm{in $\mathcal{D}'(\mathbb{R}^N \times (-\r^2,0))$}.
\end{equation}
Now, choosing $2(g-h)_+ \psi^2$ as a test function in \eqref{Kg-h} and integrating on the domain $D=\mathbb{R}^N \times [-\r^2,0]$, we get
\begin{align}\nonumber
&2\int_D \left\vert D_{m_0}\right\vert^2 \psi^2  +\boxed{4\int_D \psi(g-h)_+\langle D_{m_0}((g-h)_+),D_{m_0}(\psi)\rangle}_A -\boxed{2 \int_D (g-h)_+\psi^2 Y((g-h)_+)}_B\\ \label{boxedtermsp}
&\qquad\qquad\qquad\qquad\qquad + \boxed{2\int_D (g-h)_+\psi^2 D_{m_0} \cdot \tilde{H}_1}_C+\boxed{2\int_D (g-h)_+\psi^2 \tilde{H}_0}_D=0.
\end{align}
We estimate the boxed term $A$ by applying Young's inequality \eqref{youngA} and choosing $\varepsilon =1$. As far as we are concerned with the boxed term $B$, we rewrite it as
\begin{align*}
-\boxed{2 \int_D (g-h)_+\psi^2 Y((g-h)_+)}_B&=-\int_D \psi^2 Y_0((g-h)_+^2)+\int_D \psi^2 \partial_t((g-h)_+^2)\\
&=-\int_D [\psi^2 Y_0((g-h)_+^2)+\int_D \partial_t((g-h)_+^2\psi^2)- (g-h)_+^2 \partial_t(\psi^2)] \\
&=\int_D (g-h)_+^2 Y(\psi^2)+\int_D  \partial_t((g-h)_+^2\psi^2),
\end{align*}
where in the first line we defined $Y_0=\langle Bx, D \rangle$, in the second line we used the equality $\p_t((g-h)_+^2 )\psi^2 = \p_t((g-h)_+^2 \psi^2) - \p_t(\psi^2)(g-h)_+^2 $ and in the third we integrated by parts the term involving $Y_0$. Finally, we take care of the boxed term $C$ and $D$ using Young's inequality as follows
\begin{align*}
\boxed{2\int_D (g-h)_+\psi^2 D_{m_0} \cdot \tilde{H}_1}_C &=- 2\int_D \langle D_{m_0}((g-h)_+\psi^2), \tilde{H}_1 \rangle \\
&\leq \frac{1}{2}\Vert \psi D_{m_0}(g-h)_+ \Vert_{L^2(D)}+\frac{1}{2}\Vert (g-h)_+ D_{m_0}\psi \Vert_{L^2(D)}+10\Vert \tilde{H}_1 \Vert_{L^2(D)},\\
\boxed{2\int_D (g-h)_+\psi^2 \tilde{H}_0}_D &\leq 2\varepsilon \Vert (g-h)_+ \Vert_{L^2(D)}+\frac{1}{2\varepsilon}\Vert \tilde{H}_0 \Vert_{L^2(D)}.
\end{align*}

Combining the previous estimates, for every $T \in (-\r^2,0)$ and $\varepsilon >0 $ to be chosen later, we rewrite \eqref{boxedtermsp} as 
\begin{align*}
&\int_{-\r^2}^T \int_D  \partial_t((g-h)_+^2\psi^2)dx dt+2\int_D \left\vert D_{m_0}(g-h)_+\right\vert^2 \psi^2+\int_D  (g-h)_+^2 Y(\psi^2)\\
&\qquad \qquad \leq \frac{1}{2}\Vert D_{m_0}(g-h)_+ \psi\Vert_{L^2(D)}+10\Vert \tilde{H}_1 \Vert_{L^2(D)}+2\varepsilon \Vert (g-h)_+ \Vert_{L^2(D)}+\frac{1}{2\varepsilon}\Vert \tilde{H}_0 \Vert_{L^2(D)}.
\end{align*}
We now apply the fundamental theorem of calculus to the term involving the time derivative and we infer
\begin{align} \label{lwp1}
& \int_D  ((g-h)_+^2\psi^2)(x,T)dx+2\int_D \left\vert D_{m_0}(g-h)_+\right\vert^2 \psi^2+\int_D  (g-h)_+^2 Y(\psi^2)\\  \nonumber
&\quad\leq \frac{1}{2}\Vert D_{m_0}(g-h)_+ \psi\Vert_{L^2(D)}+10\Vert \tilde{H}_1 \Vert_{L^2(D)}+2\varepsilon \Vert (g-h)_+ \Vert_{L^2(D)}+\frac{1}{2\varepsilon}\Vert \tilde{H}_0 \Vert_{L^2(D)}. \nonumber
\end{align}
We then integrate in $T$ from $-\r^2$ to 0 and we obtain
\begin{align}  \label{lwp2}
& \Vert (g-h)_+\psi\Vert_{L^2(D)}+\r^2\int_D  (g-h)_+^2 Y(\psi^2)\\ \nonumber
&\quad\leq -\frac{3}{2}\r^2\Vert D_{m_0}(g-h)_+ \psi\Vert_{L^2(D)}+10\r^2\Vert \tilde{H}_1 \Vert_{L^2(D)}+2\r^2\varepsilon \Vert (g-h)_+ \Vert_{L^2(D)}+\frac{\r^2}{2\varepsilon}\Vert \tilde{H}_0 \Vert_{L^2(D)}\\ \nonumber
&\quad\leq  10\r^2\Vert \tilde{H}_1 \Vert_{L^2(D)}+2\r^2\varepsilon \Vert (g-h)_+ \Vert_{L^2(D)}+\frac{\r^2}{2\varepsilon}\Vert \tilde{H}_0 \Vert_{L^2(D)}. \nonumber
\end{align}
We later observe that $(g-h)_+\psi$ equals $(u-h)_+$ and $Y(\psi^2)$ equals 0 in $\Q_1$. In addition, the following estimates hold
\begin{align}  \label{lwp3}
\Vert \tilde{H}_0 \Vert_{L^2(D)} &\leq \Vert H_0 \Vert_{L^2(\Q_{ext})}+\Vert D_{m_0}\psi \Vert_{L^\infty(\Q_{ext})}\left(\Vert D_{m_0}u \Vert_{L^2(\Q_{ext})}+\Vert H_1 \Vert_{L^2(\Q_{ext})}\right),\\ \nonumber
\Vert \tilde{H}_1 \Vert_{L^2(D)} &\leq \Vert H_1 \Vert_{L^2(\Q_{ext})}+\Vert D_{m_0}u \Vert_{L^2(\Q_{ext})}. \nonumber
\end{align}
By combining \eqref{lwp1}, \eqref{lwp2} and \eqref{lwp3} and choosing $\varepsilon=\frac{1}{4\r^2}$ the claim is proved.
$\hfill \square$

\medskip

Given the local Poincaré inequality proved in Lemma \ref{weak-poincare-h}, we just need to estimate the function $h$ defined in \eqref{defpsi} in order to complete the proof of Theorem
\ref{weak-poincare}. In particular, our aim is to show that the error function $h$ is bounded from above by $\vartheta_0 M$, where $\vartheta_0 \in (0,1)$ is a constant only depending on $Q$, $\l$ and $\Lambda$. In order to prove such result, we first need to explicitly construct an appropriate cut-off function, that differs from the one considered in \cite[Lemma 3.3]{GI} due to the more involved structure of our drift term $Y$. 
We note that our construction of the suitable cut-off function is constructive, in contrast with the one proposed in \cite{GI}. 
\begin{lemma} \label{coc}
Given $\eta \in (0,1]$ and $T \in (0,\eta^2)$, there exists a smooth function $\psi_1:\mathbb{R}^{N}\times[-1-\eta^2,0]$, supported in $\lbrace (x,t) : |x_j|\leq 2^{\alpha_j}, j=1,\ldots,N, t \in [-1-\eta^2 ,0]\rbrace$, equal to 1 in $\Q_1$, and such that the following conditions hold
\begin{equation}\label{conditions}
\begin{split}
Y\psi_1 &\leq 0 \quad\textrm{everywhere}\\  
Y \psi_1 &\leq -1 \quad\textrm{if $t \in (-1-\eta^2,-1-T]$}.
\end{split}
\end{equation}
\end{lemma}
\begin{proof}
Let us consider the cut-off function $\chi \in C^\infty([0,+\infty))$ defined by
\begin{equation}\label{chi2}
\chi(s)=\left\{ \begin{array}{ll}
0,\quad &\textrm{if $s > \frac{2}{\sqrt{2}}$},\\
1,\quad &\textrm{if $0\leq s \leq C+1$},
\end{array} \right. \quad \chi'\leq 0,
\end{equation}
where $C >1$ is a constant we shall specify later on. In addition, we consider a smooth function $\chi_t : [-1-\eta^2,0]\rightarrow [0,1]$ equal to 1 in $[-1,0]$, with $\chi_t(-1-\eta^2)=0$, $\chi_t' \geq 0$ in $[-1-\eta^2,0]$ and $\chi_t'=1$ in $[-1-\eta^2,-1-T]$. Now, setting 
\begin{eqnarray*}
\chi_0(x,t)=\chi\left(\sum_{j=m_0}^N \frac{2 x_j^2}{2^{2\alpha_j}\sqrt{2}}-C \, t\right),
\end{eqnarray*}
we define the cut-off function $\psi_1$ as follows
\begin{eqnarray*}
\psi_1(x,t)=\chi(\Vert (x_1,\ldots,x_{m_0})\Vert)\chi_0(x,t)\chi_t(t).
\end{eqnarray*}
We only have to check that conditions \eqref{conditions} hold, as the other desired properties immediately follow from the definition of $\psi_1$. To this end, we compute the following derivative
\begin{eqnarray*}
Y\chi_0=\chi'((...))\left[\sum_{i=1}^N\sum_{j>m_0}2x_ib_{ij}x_j2^{-2\alpha_j-1/2}+C\right],
\end{eqnarray*}
where $(...)$ denotes $\left(\sum_{j=m_0}^N \frac{2 x_j^2}{2^{2\alpha_j}\sqrt{2}}-C \, t\right)$. It can be shown (see \cite{WZ3}) that it is well defined a certain constant $C>1$ such that
\begin{eqnarray*}
C\geq \sum_{i=1}^N\sum_{j>m_0}2x_ib_{ij}x_j2^{-2\alpha_j-1/2}.
\end{eqnarray*}
Thus, with such a choice of $C$ and keeping in mind that $\chi_t' \geq 0$ in $[-1-\eta^2,0]$ and $\chi_t'=1$ in $[-1-\eta^2,-1-T]$, we have
\begin{align*}
&Y\psi_1= \chi \,\chi_t \, Y \chi_0 - \chi\, \chi_0 \, \chi'_t  \leq 0 \qquad\textrm{everywhere},\\
&Y\psi_1= \chi \,\chi_t \, Y \chi_0 - \chi\, \chi_0 \, \chi'_t  \leq -1 \quad\textrm{if $t \in (-1-\eta^2,-1-T]$}.
\end{align*}
\end{proof}
Thus, we are now in a position to state and prove the following result regarding the control of the localization term $h$
defined in \eqref{defpsi}.
\begin{lemma} \label{localizationlem}
Let $\eta \in (0,1]$ and let $\Q_{ext}$ be as defined in \eqref{Qzero}. then there exist $R=R(Q,\eta)>1$, $\vartheta_0=\vartheta_0(Q,\eta) \in (0,1)$ and a $C^\infty$ cut-off function $\psi: \mathbb{R}^{N+1} \rightarrow [0,1]$, with support in $\Q_{ext}$ and equal to 1 in $\Q_1$, such that for all $u \in \W$ non-negative bounded functions defined on $\Q_{ext}$,
 the function $h$ solution to the Cauchy problem \eqref{defpsi} with $\r^2=1+\eta^2$ satisfies
\begin{equation}\label{est-h}
h \leq \vartheta_0 \Vert u \Vert_{L^\infty(\Q_{ext})}.
\end{equation}
\end{lemma} 
\begin{proof}
We assume that $u$ is not identically vanishing in $\Q_{ext}$. Indeed, if $u=0$ in $\Q_{ext}$, then $h=0$ and inequality \eqref{est-h} is trivially satisfied. Moreover, we can reduce to the case of a function $u$ with $L^\infty$- norm equal to 1 by taking $u/\Vert u \Vert_{L^\infty(\Q_{ext})}$.

We now fix $T=\eta^2/2$; then, $$|\Q_{zero}\cap \lbrace t \leq -1-T \rbrace|=\frac{1}{2}|\Q_{zero}|.$$
We now consider the cut-off function
\begin{eqnarray*}
\psi(x,t)=\psi_1(x/R,t),
\end{eqnarray*}
whew $R>1$ is a constant we will specify later and $\psi_1$ is given by Lemma \ref{coc}. We observe that, by definition of $\psi_1$, $\psi$ is supported in $\Q_{ext}$ and equal to 1 in $\lbrace (x,t) : |x_j|\leq R, j=1,\ldots,N, t \in (-1 ,0]\rbrace$. In addition, it satifies
\begin{eqnarray*}
\K \psi(x,t)=R^{-2}\Delta_{m_0}\psi_1(x/R,t)+Y\psi_1(x/R,t).
\end{eqnarray*}
Thus, in virtue of \eqref{defpsi}, we have
\begin{eqnarray*}
\K (h-\psi)=\frac{u-1}{R^2}\Delta_{m_0}\psi_1(x/R,t)+(u-1)Y\psi_1(x/R,t)
\end{eqnarray*}
and we can write the difference $h-\psi$ as
\begin{eqnarray}\label{diffhpsi}
h-\psi = E_R +N_R,
\end{eqnarray}
where $E_R$ and $N_R$ are solutions in $\mathbb{R}^N \times (-1-\eta^2,0) $ to the following Cauchy problems
\begin{align*}
\K E_R &=\frac{u-1}{R^2}\Delta_{m_0}\psi_1(x/R,t),\\
\K N_R &=(u-1)Y\psi_1(x/R,t),
\end{align*}
with $E_R=P_R=0$ at the initial time. We first focus on the term involving $E_R$, and we remark that 
\begin{eqnarray}\label{stimaER}
\K E_R \leq \frac{ C'}{ R^{2}},
\end{eqnarray}
where $C'=\Vert \Delta_{m_0}\psi_1 \Vert_{L^\infty}$ is a constant only depending on $Q$, $\lambda$, $\Lambda$ and $\eta$. As far as the term involving $N_R$ is concerned, we observe that, owing to $Y\psi_1 \leq -1 $ for $t \in (-1-\eta^2,-1-T)$, we have
\begin{eqnarray*}
\K N_R \leq -\mathbb{I}_{\mathcal{Z}}, \quad \textrm{in $\mathbb{R}^N \times (-1-\eta^2,0) $},
\end{eqnarray*}
where $\mathcal{Z}:= \Q_{zero} \cap \lbrace t \leq -1-T \rbrace$. Let $N$ such that $\K N=-\mathbb{I}_{\mathcal{Z}}$ in $\mathbb{R}^N \times (-1-\eta^2,0) $ and $N=0$ at the initial time $t=-1-\eta^2$. Then, the maximum principle \cite{Bony69} for the principle part operator $\K$ yields
\begin{eqnarray*}
P \leq N_R, \quad \textrm{in $\Q_1$}.
\end{eqnarray*}
We now represent $N$ in using the fundamental solution $\Gamma$ of $\K$ and we infer
\begin{eqnarray*}
N(z)=\int \Gamma(z,\zeta) \left(-\mathbb{I}_{\mathcal{Z}}(\zeta)\right) d\zeta \leq -\frac{1}{2} m |\Q_{zero}|=:-\delta_0,
\end{eqnarray*}
where $m=\min_{\Q_1 \times \Q_{zero}\cap\lbrace t \leq -1-T \rbrace}\Gamma(z,\zeta)$. As a consequence,
\begin{eqnarray}\label{stimaNR}
N_R \leq -\delta_0,
\end{eqnarray}
for a constant $\delta_0$ only depending on $Q$ and $\eta$. Using estimates \eqref{stimaER} and \eqref{stimaNR} in \eqref{diffhpsi}, we finally obtain
\begin{eqnarray*}
h \leq 1 -\delta_0 + \frac{C'}{R^2}.
\end{eqnarray*}
We now observe that for $R$ large enough we have $C/R^2 \leq \delta_0/2$. Thus, setting $\vartheta_0=1-\delta_0/2 <1$, we get the desired inequality \eqref{est-h}.
\end{proof}

\setcounter{equation}{0}\setcounter{theorem}{0}
\section{Main results} 
\label{harnack}
This section is devoted to the proof of our main results. The approach we present here is an extension of the 
method inspired by \cite{Kruz1,Kruz2} and then followed by Guerand and Imbert in \cite{GI} for the kinetic Kolmogorov-Fokker-Planck equation. 
In particular, we remark that an analogous approach based on a weak Poincaré type inequality was firstly 
introduced by Wang and Zhang in \cite{WZ4, WZ3} for the Kolmogorov equation $\L u = 0$ under the assumption 
$Y u \in L^{2}$, and thus with a stronger notion of weak solution, and with a different $\log-$transform. 
The main advantage of our approach is that it only relies on 
the structure of the function space $\W$ to which every weak solution belongs to and on the non-Euclidean 
geometrical structure presented in Section \ref{preliminaries} behind the operator $\L$. To our knowledge, 
it is the first time that the study of the weak regularity theory is carried on replacing the classical $L^{Q+2}$ integrability assumptions for the lower order coefficients $b$, $c$ and the non zero right-hand 
side $f$ with \textbf{(H3)}.

\subsection{Weak Harnack inequality}
First of all we address the proof of the weak Harnack inequality (Theorem \ref{weak-harnack}) that relies on 
combining the fact that super-solutions to \eqref{defL} expand positivity along times (Lemma \ref{expand-pos}) with 
the covering argument presented in Appendix \ref{stacked}. 

The derivation of the weak Harnack inequality in the present paper from the expansion of positivity follows very 
closely the reasoning in \cite{IS-weak} and extends the results presented in \cite{GI} for the Kolmogorov-Fokker-
Planck case. For reader's convenience, we here state (and adapt to our more involved case) the results contained in \cite[Section 4]{GI}, sketching their proofs only when they differ from the ones contained in the aforementioned paper.

We observe that, in contrast with parabolic equations, it is not possible to apply a classical Poincaré inequality
in the spirit of \cite{M1}. Indeed, in our case there is a positive quantity replacing the avarage in the usual Poincar\`{e} inequality (see the statement of Theorem \ref{weak-poincare}). Following \cite{GI} we circumvent this difficulty by estabilishing a weak expansion of positivity of super-solutions to \eqref{defL}. More precisely, given a small cylider $\Q_{pos}$ lying in the past of $\Q_1$ (see Definition \eqref{Qexth}), we show that the positivity of a non-negative super solution $u$ lying above 1 in a "big" part of $\Q_{pos}$ is spread to the whole $\Q_1$ (see Lemma \ref{expand-pos}). In other words, a positivity in measure in a smaller cylinder $\Q_{pos}$ is transformed into a pointwise positivity in a bigger cylinder $\Q_1$. We emphasize that such a weak expansion of positivity was already proved in \cite{GIMV} thanks to an intermediate value lemma, following De Giorgi's original proof, but the argument was not constructive and specific for the Fokker-Planck equation case. 

Lastly, we mention that Moser \cite{M1} and Trudinger \cite{Trudinger} proved a weak Harnack inequality in the spirit of Theorem \ref{weak-harnack} in the setting of parabolic equations. We also recall that Di Benedetto and Trudinger \cite{dibenedetto} proved that non-negative functions in the elliptic De Giorgi class, which corresponds to super-solutions to elliptic equations, satisfy a weak Harnack inequality. Moreover, let us emphasize that in \cite{GWang} it is proved a weak Harnack for the corresponding parabolic case, i.e. for functions in the parabolic De Giorgi's class. We conclude by observing that quantitative interior H\"{o}lder regularity estimates for functions in the parabolic De Giorgi class (and for parabolic equations with rough coefficients) can be found in \cite{guerand}.

\medskip

We first study how equation \eqref{defL} spreads positivity of super-solutions. More precisely, we state the upcoming Lemma \ref{expand-pos}, given in terms of the cylinders
\begin{align}\label{Qexth}
	\Q_{pos}&=B_{\theta}\times B_{\theta^3}\times\ldots \times B_{\theta^{2\kappa+1}}\times (-1-\theta^2,-1],\\ \nonumber
	\tilde{\Q}_{ext}&=B_{3R}\times B_{3^3R}\times\ldots \times B_{3^{2\kappa+1}R}\times (-1-\theta^2,0],
\end{align}
where $R=R(\theta,Q,\lambda,\Lambda)$ is the constant given by Lemma \ref{localizationlem} and $\theta \in (0,1]$ is a parameter we will choose later on. In particular, $\theta$ will be chosen such that the stacked cylinder $\overline \Q_{pos}^{m}$ (see definition \ref{qm}) is contained in $\Q_1$. To stress the dependence of $R$ on $\theta$, we will sometimes write $R_\theta$ instead of $R$.

\begin{lemma}\label{expand-pos}
Let $\theta \in (0,1]$ and $\Q_{pos}$, $\tilde{\Q}_{ext}$ be the cylinders defined in \eqref{Qexth}. Then there exist a small constant $\eta_0=\eta_0(\theta,Q,\lambda,\Lambda)\in (0,1)$ such that for any non-negative super-solution $u$ of \eqref{defL} in some cylindrical open set $\Omega \supset \tilde{\Q}_{ext}$ such that $\left\vert \lbrace u \geq 1 \rbrace \cap \Q_{pos}\right\vert \geq \frac{1}{2}\left\vert \Q_{pos}\right\vert$, we have $u \geq \eta_0$ in $\Q_1$.
\end{lemma}
\begin{proof}
For the sake of completeness we hereby state a sketch of the proof, that is an extension of \cite[Lemma 4.1]{GI}.
We consider $g=G(u+\epsilon)$, where $G$ is the convex function defined in \cite[Lemma 2.1]{GI}. In particular, 
$G$ is such that 
\begin{itemize}
	\item $G'' \ge (G^{'})^{2}$ and $G^{'} \le 0$ in $]0, + \infty[$,
	\item $G$ is supported in $]0,1]$, 
	\item $G(t) \sim - \ln t$ as $t \to 0^{+}$,
	\item $- G^{'}(t) \le \frac1t$ for $t \in ]0, \frac14]$. 
\end{itemize}
Thus, we have that $|G'(u+\epsilon)|\leq |G'(\epsilon)|\leq \epsilon^{-1}$ as $u$ is non-negative. Moreover, adapting \cite[Lemma 2.2]{GI} to our case, we find that $g$ is a non-negative sub-solution to \eqref{defL} with $f$ replaced by $fG'(u+\epsilon)$. This implies in particular that the drift term $Yg$ is bounded, i.e. $Y g \leq \div (A D g)+\langle b, D g \rangle +cg+\varepsilon^{-1}\Vert f \Vert_{{L^q}(\tilde{\Q}_{ext})}$. The rest of the proof follows very closely the one of \cite[Lemma 4.1]{GI}, with the only difference that we consider our Theorem \ref{boundedness} instead of the classical $L^2-L^\infty$ estimate and the weak Poincar\`{e} inequality \ref{weak-poincare} instead of \cite[Theorem 1.4]{GI}.
\end{proof}
As a straightforward consequence of Lemma \ref{expand-pos} we have the following result, whose proof is obtained reasoning exactly as in \cite[Lemma 4.2]{GI}.
\begin{lemma}\label{expositivy2}
Let $m\geq 3$ and let $R$ be the constant given in Lemma \ref{expand-pos} for $\theta\leq m^{-1/2}$. Then there exists a constant $M=M(m,Q,\lambda,\Lambda) >1$ such that for any non-negative super-solution 
$u$ to \eqref{defL} with $f$ equal to 0 satisfying $\left\vert \lbrace u \geq M \rbrace \cap \Q_1 \right\vert \geq \frac{1}{2}\vert \Q_1\vert$, we have $u \geq 1$ in $\overline{\Q}_1^m$ (see \eqref{qm}).
\end{lemma}

Before proving the weak Harnack inequality, we need to show that we can spread positivity along "suitable" cylinders. More precisely, recalling the definition of the open ball in \eqref{openball}, we set 
\begin{equation}\label{Qmeno}
	\Q_+=B_{\omega}\times B_{\omega^3}\times\ldots \times B_{\omega^{2\kappa+1}}\times (-\omega^2,0], \quad \Q_-=B_{\omega}\times B_{\omega^3}\times\ldots \times B_{\omega^{2\kappa+1}}\times (-1,-1+\omega^2], 
\end{equation}
where $\omega$ is a small positive constant. In particular, we will choose $\omega$ small enough so that, when expanding positivity from a given cylinder $\Q_r(z_0)$ in the past, the union of the stacked cylinders where the positivity is spread includes $\Q_+$. Moreover, we will choose the radius $R_0$ in the statement of Theorem \ref{weak-harnack} so that Lemma \ref{expand-pos} can be applied to every stacked cylinder. The two previous statements are specified in Lemma \ref{stackedcylinders}. 
The staking cylinders Lemma \ref{stackedcylinders}, combined with Lemma \ref{expand-pos}, implies the following result regarding the expansion of positivity for large times. 

In the sequel, we will largely use the cylinders $\Q_r[k]$, for $k=1,\ldots,N$ and $\Q_{R_{N+1}}[N+1]$, whose definition and properties are presented in Appendix \ref{stacked}. 
\begin{lemma}
\label{lem-largetimes}
Let $R_{1/2}$ be the constant given by Lemma \ref{expand-pos} for $\theta=1/2$ and let $u$ be any non-negative super-solution to \eqref{defL} with $f=0$ in $\O \supset \Q$ such that $\left\vert \lbrace u \geq M \rbrace \cap \Q_r(z_0)\right\vert \geq \frac{1}{2}\left\vert \Q_r(z_0) \right\vert$ for some $M>0$ and for some cylinder $\Q_r(z_0) \subset \Q_-$. Then there exists a positive constant $p_0$, only depending on $Q$, $\lambda$, $\Lambda$, such that
\begin{eqnarray}
u \geq M\left( \frac{r^2}{4} \right)^{p_0}, \quad \textit{in $\Q_+$}.
\end{eqnarray}
\end{lemma}
\begin{proof}
We apply Lemma \ref{expand-pos} for $\theta=\frac{1}{2}$ to the function $u/M$, with $\Q_r (z_0)$ and $\Q_r[1]$ taking the role of $\Q_{pos}$ and $\Q_1$ (this is achieved through a rescaling argument) and obtain $u/M\geq \eta_0$ in $\Q_r[1]$. We then apply it to $u/(M\eta_0)$ and get $u\geq M\eta_0^2$ in $\Q_r[2]$. Reasoning by induction on $k=1,\ldots,N$ we infer $u\geq M\eta_0^k$ in $\Q[k]$.

By exploiting Lemma \ref{expand-pos} again, we get $u\geq M\eta_0^{N+1}$ in $\Q_{R_{N+1}}[N+1]$, which implies that the same inequality holds true in $\Q_+$. As $T_N \leq -t_0 < 1$, we have in particular $4^N r^2 \leq 1$. Picking $p_0 >0$ so that $\eta_0=\left(\left(\frac{1}{4}\right)^\frac{N}{N+1}\right)^{p_0}$, we finally obtain 
\begin{equation*}
u \geq M \left( \left(\left(1/4\right)^\frac{N}{N+1}\right)^{N+1}\right)^{p_0}=M \left( \left( 1/4\right)^N\right)^{p_0} \geq M \left(r^2/4\right)^{p_0},
\end{equation*}
which concludes the proof.
\end{proof}

From now on we will assume $\omega<1/\sqrt{2\kappa+1}$, where $\k$ is defined in \eqref{split.coord.RN}.
We are in a position to prove the main result of this Section, Theorem  \ref{weak-harnack}. 

\medskip 

\textit{Proof of Theorem \ref{weak-harnack}.}
We start the proof by fixing the parameters $\omega$ and $R_0$ in order to select the appropriate geometric setting. More precisely, we choose $\omega$ so that we capture $\Q_+ $ when applying Lemma \ref{stackedcylinders}, namely we fix $\omega<\frac{1}{\sqrt{2\kappa+1}}$. In addition, we choose the radius $R_0$ so that the stacked cylinders do not leak out of $\Q^0$, i.e. $R_0 \geq 6\left(2\kappa+1\right)R_{1/2}$, where $R_{1/2}$ is the constant given by Lemma \ref{expand-pos} when $\theta=1/2$. As we want to apply Lemma \ref{expositivy2} to cylinders contained in $\Q_-$, we also assume $R_0\geq 3(2\kappa+1)R_{m^{-1/2}}m^{(2\kappa+1)/2}\omega^{2\kappa+1}$, where $R_{m^{-1/2}}$ is the constant given by Lemma \ref{expand-pos} for $\theta=m^{-1/2}$. 

Our aim is to reduce ourselves to the case where
\begin{eqnarray} \label{hp1}
\inf_{\Q_+}u\leq 1, \quad \textit{and} \quad f=0.
\end{eqnarray}
On one hand, if $\inf_{\Q_+}u >1$ we can simply consider $\bar{u}=u/\left(\inf_{\Q_+}u+1\right)$ and reduce to the case where $\inf_{\Q_+}u\leq 1$. 
On the other hand, if $f\ne0$ and $c=0$ we have that $\tilde{u}:=u+\vartheta t \Vert f \Vert_{L^q(\Q^0)}$ is a super-solution to equation \eqref{defL} with source term equal to 0, provided that we choose $\vartheta$ such that 
\begin{equation}
	\label{theta-p}
	\vartheta=R_0^{-\frac{Q+2}{q}}.
\end{equation} 
Indeed, exploiting the fact that $u$ is non-negative super-solution to \eqref{defL} in $\Q^0$, we infer
\begin{align}\label{kolmo-sub2} \nonumber
	   &\int_{\Q^0} - \langle A D\tilde{u}, D\phi \rangle + \phi Y \tilde{u} + \langle b , D\tilde{u} \rangle \phi \\
	   &\quad= \int_{\Q^0} - \langle A Du, D\phi \rangle + \phi Y u -\phi\vartheta\Vert f \Vert_{L^q(\Q^0)}+ \langle b , Du \rangle \phi \\ \nonumber
	   &\quad \leq \int_{\Q^0} f \phi-\int_{\Q^0} \phi\vartheta\Vert f \Vert_{L^q(\Q^0)}.
	\end{align}
We now observe that the last line in \eqref{kolmo-sub2} can be estimated as follows
\begin{align*}
\int_{\Q^0} f \phi-\int_{\Q^0} \phi\vartheta\Vert f \Vert_{L^q(\Q^0)} &\leq \left(R_0^{Q+2}\right)^{1-1/q}\Vert f \Vert_{L^q(\Q^0)}\Vert \phi \Vert_{L^\infty(\Q^0)}-\vartheta\Vert f \Vert_{L^q(\Q^0)}R_0^{Q+2}\Vert \phi \Vert_{L^\infty(\Q^0)}\\
&=\Vert f \Vert_{L^q(\Q^0)}R_0^{Q+2}\Vert \phi \Vert_{L^\infty(\Q^0)}\left( \left(R_0^{Q+2}\right)^{1-1/q}-\vartheta R_0^{Q+2}\right),
\end{align*}
which is equal to 0 when $\vartheta=R_0^{-\frac{Q+2}{q}}$. Thus, $\tilde{u}$ is a super-solution of equation \eqref{defL} with $f=0$ and the weak Harnack inequality for $\tilde{u}$ implies the one for $u$. 
Lastly, if $c \ne 0$ the reasoning follows by replacing the value of $\vartheta$ in 
\eqref{theta-p} by  
\begin{equation*}
	\vartheta = \frac{1}{(R_{0}^{Q+2} )^{1/q} - \| c \|_{L^{q}(\Q^{0})}}.
\end{equation*}

We now want to prove that for all $k \in \mathbb{N}$, the following inequality holds
\begin{eqnarray}\label{induction}
\left\vert \lbrace u > M^k \rbrace \cap \Q_1 \right\vert \leq \tilde{C}(1-\tilde{\mu})^k,
\end{eqnarray}
for some constants $\tilde{\mu} \in (0,1)$ $M>1$ and $\tilde{C}>0$ that only depend on $Q$, $\lambda$ and $\Lambda$. The proof of this fact is carried out by induction. For $k=1$ it is sufficient to choose $\tilde{\mu} \leq \frac{1}{2}$ and $\tilde{C}$ such that $|\Q_-|\leq \frac{1}{2}\tilde{C}$. We now assume that \eqref{induction} holds true for $k\geq 1$ and we prove it for $k+1$. To this end, we consider the sets
\begin{eqnarray}
E:=\lbrace u >M^{k+1}\rbrace \cap \Q_- , \qquad F:=\lbrace u >M^{k}\rbrace \cap \Q_1.
\end{eqnarray}
We observe that $E$ and $F$ satisfy the assumptions of Corollary \ref{ink-spot-cor} with $\Q_1$ replaced by $\Q_-$ and $\mu=1/2$. Indeed, by definition $E$ and $F$ are bounded measurable sets such that $E \subset \Q_- \cap F$. We now consider a cylinder $\Q=\Q_r(z)\subset \Q_-$ such that $|\Q \cap E| >\frac{1}{2}|\Q|$, i.e.
\begin{eqnarray*}
|\lbrace u > M^{k+1}\rbrace \cap \Q|>\frac{1}{2}|\Q|.
\end{eqnarray*}
We show that $r$ needs to be small, that is to say $r$ is less than some parameter $r_0=r_0(Q,\lambda,\Lambda,k)$. Indeed, applying Lemma \ref{lem-largetimes} to $u$, we obtain $u\geq M^{k+1}(r^2/4)^{p_0}$ in $\Q_+$. Thus, owing to $\inf_{\Q_+}u \leq 1$, we infer $1 \geq M^{k+1}(r^2/4)^{p_0}$ and therefore it is sufficient to choose $r_0\leq2 M^{-k-1/2p_0}$. In order to apply Corollary \ref{ink-spot-cor}, we are left with proving that $\overline{\Q}^m \subset F$, which holds true if $\overline{\Q}^m \subset \lbrace u > M^k \rbrace$. To this end, we apply Lemma \ref{expositivy2} to $u/M^k$ after rescaling the cylinder $\Q$ in $\Q_1$.

In virtue of Corollary \ref{ink-spot-cor}, there exist $c_{\rm is} \in (0,1)$ and $C_{\rm is}>0$ such that
\begin{align*}
|E|=\left\vert \lbrace u >M^{k+1}\rbrace \cap \Q_- \right\vert &\leq \frac{m+1}{m}\left( 1- \frac{c_{\rm is}}{2}\right)\left( \left\vert \lbrace u >M^{k}\rbrace \cap \Q_1\right\vert +C_{\rm is} mr_0^2 \right)\\
&\leq \left(1-\frac{c_{\rm is}}{4}\right)\left( \left\vert \lbrace u >M^{k}\rbrace \cap \Q_1\right\vert +C_{\rm is} mr_0^2 \right),
\end{align*}
provided that we chose $m \in \mathbb{N}$ so that $\frac{m+1}{m}\left( 1- \frac{c_{\rm is}}{2}\right)\leq 1-\frac{c_{\rm is}}{4}$. Thanks to the induction assumption and our choice of $r_0$ we get
\begin{align*}
|E|&\leq \left(1-\frac{c_{\rm is}}{4}\right)\left( \tilde{C}(1-\tilde{\mu})^k +C_{\rm is}mr_0^2 \right)\\
&\leq \left(1-\frac{c_{\rm is}}{4}\right)\left( \tilde{C}(1-\tilde{\mu})^k +C_{\rm is} mM^{-\frac{k+1}{p_0}} \right).
\end{align*}
Picking then $\tilde{\mu}$ small enough so that $M^{-1/p_0}\leq (1-\tilde{\mu})$ and $\tilde{\mu}\leq \frac{c_{\rm is}}{4}$, we obtain
\begin{align*}
|E|&\leq \tilde{C}\left(1-\frac{c_{\rm is}}{4}\right)(1-\tilde{\mu})^k\left( 1 +\tilde{C}^{-1}mM^{-\frac{1}{p_0}} \right)\\
&\leq \tilde{C}(1-\tilde{\mu})^{k+1}\left( 1 +\tilde{C}^{-1}mM^{-\frac{1}{p_0}} \right).
\end{align*}
Picking $\tilde{C}$ large enough so that $\left( 1 +\tilde{C}^{-1}mM^{-\frac{1}{p_0}} \right)\leq 2$ we conclude the proof of \eqref{induction}. By extending estimate \eqref{induction} to the continuous case (i.e. $k \in \mathbb{R}$ and $k\geq 1$) and applying the layer cake formula to $\int_{\Q_- }f^p$ for some exponent $p$, we obtain that $\int_{\Q_-} f^p$ is bounded from above by a constant that only depends on $Q$, $\lambda$ and $\Lambda$. 
$\hfill \square$

\subsection{Harnack inequality and local H\"older continuity} \hfill \\

\smallskip

\textit{Proof of Theorem \ref{harnack-thm}.}
The full Harnack inequality is a direct 
consequence of the combination of the local boundedness of weak sub-solutions proved in Theorem 
\ref{boundedness} and the weak Harnack inequality of Theorem \ref{weak-harnack}. $\hfill \square$

\medskip

\textit{Proof of Theorem \ref{local-holder}.} The H\"older continuity of weak solutions is classically obtained by proving that the oscillation of the solution decays 
by a universal factor. This can be achieved in two different ways. Either by applying Lemma \ref{expand-pos} with $\theta=1$ in the same spirit of \cite[Appendix B]{GI}, or by directly applying the weak Harnack inequality, Theorem \ref{weak-harnack}, following a standard argument, for further reference see \cite{GT}.
$\hfill \square$

	\appendix
	\setcounter{theorem}{0}
	\counterwithin{theorem}{section}
	\counterwithin{equation}{section}
\section{The Ink-Spots Theorem}
\label{ink-spot}
For the sake of completeness, we provide here the proof of the Ink Spots Theorem for the case of ultraparabolic equations. This theorem involves a covering argument in the spirit of Krylov and Safonov \cite{KrylovSafonov} growing ink spots theorem, or the Calder\'on-Zygmund 
decomposition, and it is a fundamental ingredient for the proof of the weak Harnack inequality (see Theorem 
\ref{weak-harnack}). In order to give its statement in our setting, we introduce the delayed cylinder
\begin{align}
	\label{qm}
	\overline \Q_{r}^{m}(z_{0}) &= \left( (0, \ldots, 0, mr^{2}) \circ \Q_{r}(z_{0}) \right) 
	\cap \left(  \R^{N+1} \times (t_{0}, + \infty) \right) 
\end{align}
where $z_{0}= (x_{0},t_{0}) = (x^{(0)}_{0}, \ldots, x^{(\k)}_{0}, t_{0}) \in \R^{N+1}$. We remark that 
$\overline \Q_{r}^{m}(z_{0})$ starts immediately at the end of $\Q_{r}(z_{0})$, with which shares the same
values for $x^{(0)}$, and its structure follows the non Euclidean geometry presented in Section \ref{preliminaries} associated to the principal part operator $\K$. The aim of this section is to prove the following statement. 
\begin{theorem}
	\label{ink-spot1}
	Let $E \subset F$ be two bounded measurable sets. We assume there exists a constant 
	$\mu \in ]0,1[$ such that
	\begin{itemize}
		\item $E \subset \Q_{1}$ and $|E| < (1 - \mu) |\Q_{1}|$;
		\item moreover, there exist an integer $m$ such that for any cylinder
			$\Q \subset \Q_{1}$ such that $\overline \Q^{m} \subset \Q_{1}$ and 
			$|\Q \cap E| \ge (1-\mu)|\Q|$, we have that $\overline \Q^{m} \subset F$.
	\end{itemize}
	Then for some universal constant $c_{\rm is} \in (0,1)$
	only depending on $N$, there holds
	\begin{equation*}
	|E| \le \frac{m+1}{m}(1 - c_{\rm is}\mu) |F| .
	\end{equation*}
\end{theorem}

\begin{remark}
 Theorem \ref{ink-spot1} still holds true if we replace $\Q_1$ with $\Q_-$ defined in \eqref{Qmeno}.
\end{remark}
As it has already been pointed out by Imbert and Silvestre in \cite{IS}, there is no chance to adapt the 
Calder\'on-Zygmund decomposition to this context, because it would require to split a larger piece into smaller ones of the same type and this is impossible due to the non Euclidean nature of our geometry. What we do is a generalization of the procedure proposed in \cite{IS}, that is in fact an adaptation of the growing ink-spots theorem, whose original construction in the parabolic case dates back to Krylov and Safonov \cite[Appendix A]{KrylovSafonov}. 

Moreover, when we need to confine both $E$ and $F$ to stay within a fixed cylinder, the following corollary directly follows. 
\begin{corollary}
	\label{ink-spot-cor}
	Let $E \subset F$ be two bounded measurable sets. We assume 
	\begin{itemize}
		\item $E \subset \Q_{1}$;
		\item there exist two constants $\mu , r_{0} \in ]0,1[$ and an integer $m$ 
		such that for any cylinder $\Q \subset \Q_{1}$ of the form $Q_{r}(z_{0})$ 
			such that $|\Q \cap E| \ge (1-\mu)|\Q|$, we have $\overline \Q^{m} \subset F$
			and also $r < r_{0}$.
	\end{itemize}
	Then for some universal constants $c_{\rm is}$ and $C_{\rm is}$ only depending on $N$
	\begin{equation*}
		|E| \le \frac{m+1}{m}(1 - c_{\rm is}\mu) \left( |F \cap \Q_{1}| + C_{\rm is} m r^{2}_{0} \right).
	\end{equation*} 
\end{corollary}

\subsection{Stacked cylinders}
First of all we recall some important properties of the following family of stacked cylinders
\begin{align*}
	k\Q_{r} = \left( 0, \ldots, 0, \frac{k^{2} - 1}{2} r^{2}  \right) \circ \Q_{kr}
		  \quad \text{and} \quad 
	k\Q_{r} (x_{0}, t_{0} )= \left( 0, \ldots, 0, \frac{k^{2} - 1}{2} r^{2}  \right) \circ \Q_{kr}(x_{0}, t_{0} ),	   
\end{align*}
where $(x_{0}, t_0) \in \R^{N+1}$, 
that are defined starting from the unit cylinder \eqref{unitcylind} for a certain $k>0$.
By definition, it is clear that $|k \Q_{r} (x_{0}, t_{0} )| = k^{Q+2} |\Q_{r} (x_{0}, t_{0} )|$,
and that the cylinders $\Q_{r} (x_{0}, t_{0} )$ are not the balls of any metric.
Thus, the important properties of the cylinders are explicitly given by the following lemmas.
\begin{lemma}
	\label{l1app}
	Let $\Q_{r_{0}} (x_{0}, t_{0}) $ and $\Q_{r_{1}} (x_{1}, t_{1} )$ be two cylinders with non empty 
	intersection, with $(x_{0}, t_{0}), (x_{1}, t_{1} ) \in \R^{N+1}$ and $2r_{0} \ge r_{1} > 0$. Then
	\begin{equation*}
		\Q_{r_{1}} (x_{1}, t_{1} ) \subset k\Q_{r_{0}} (x_{0}, t_{0}) 
	\end{equation*} 
	for some universal constant $k$.
\end{lemma} 
\begin{proof}
	Without loss of generality, we may assume $(x_{0}, t_{0}) = (0,0)$. Then we need to choose the constant 
	$k$ in order to satisfy our statement. In particular, if we consider the ball associated to the first $m_{0}$ 
	variables we get that $B_{r_{1}}(x_{1}^{(0)}) \subset B_{k r_{0}}$ if 
	\begin{equation*}
		k r_{0} \ge r_{0} + 2 r_{1} \quad \implies \quad k \ge 5 .
	\end{equation*}
	By repeating the same argument for all the $\k$ blocks of variables, we get that $k$ must satisfy the 
	following conditions: 
	\begin{equation*}
		k^{2j+1} \ge 1 + 2 \cdot 2^{2j + 1} \qquad \text{for  } j = 0, \ldots, \k.
	\end{equation*}
	As far as we are concerned with the condition regarding the time interval, we need $k$ to be such that
	\begin{equation*}
		- \frac{k^{2}+1}{2} r^{2}_{0} \le - r_{0} - 2 r_{1}^{2} \quad \implies 
		\quad k^{2} \ge 9.
	\end{equation*}
	All of these inequalities are satisfied when the first one, i.e. the one corresponding to $j=0$, is satisfied. 
	We choose $k$ to be the smallest parameter satisfying these inequalities. 
\end{proof}

\begin{lemma}
	\label{vitali}
	Let $\{ \Q_{j} \}_{j \in J}$ be an arbitrary collection of slanted cylinders with bounded radius. Then 
	there exists a disjoint countable subcollection $\{ \Q_{j_{i}} \}_{i \in I}$ such that
	\begin{equation*}
		 \bigcup \limits_{j \in J} \Q_{j} = \bigcup \limits_{i=1}^{\infty} k\Q_{j_{i}}.
	\end{equation*}
\end{lemma} 
\noindent
The proof of Lemma \ref{vitali} is the same as the classical proof of the Vitali covering lemma, where 
we employ Lemma \ref{l1app} instead of the fact that in any metric space $B_{r_{1}}(x_{1}) \subset 
5 B_{r_{0}}$, if $B_{r_{1}} (x_{1}) \cap B_{r_{0}} \ne \emptyset$ and $r_{1} \le 2 r_{0}$.

\subsection{A generalized Lebesgue differentiation theorem}
For the readers convenience, we also recall the definition of maximal function: 
\begin{equation*}
	M f ( x,t) = \sup \limits_{\Q : (x,t) \in \Q} \frac{1}{|\Q|} \int \limits_{\Q \cap \Omega} |f(y,s)| \, dy \, ds ,
\end{equation*}
where the supremum is taken over cylinders of the form $(y,s) + R\Q_{1}$. 
\begin{lemma}
	\label{max-fun}
	For every $\l > 0$ and $f \in L^1(\O)$ , we have
	\begin{equation*}
		| \{ Mf < \l \} \cap \Omega | \le \frac{C}{\l} \| f \|_{L^{1}(\O)}.
	\end{equation*}
\end{lemma}
\begin{proof}
	Let us consider $(x,t) \in \{ Mf < \l \} \cap \Omega$. Then there exists a cylinder $\Q$ such that 
	$(x,t) \in \Q$ and 
	\begin{equation*}
		\int \limits_{\Q \cap \Omega} |f(y,s)| \, dy \, ds \ge \frac{\l}{2} |\Q \cap \O|.
	\end{equation*}
	Then $\{ Mf < \l \} \cap \Omega$ is covered with cylinders $\{ \Q_{j} \}$ such that the previous inequality 
	holds. From Lemma \ref{vitali}, there exists a disjoint countable subcollection $\{ Q_{j_{i}} \}$ so that
	\begin{equation*}
		\{ Mf < \l \} \cap \Omega = \bigcup \limits_{j=1}^{\infty} Q_{j} \subset 
		\bigcup \limits_{i=1}^{\infty} kQ_{j_{i}},
	\end{equation*}
	for some integer $k$. Thus, we get
	\begin{align*}
		\| f \|_{L^{1}(\O)} \ge \int \limits_{\O \cap \cup_{i} Q_{j_{i}} } |f|
		\ge \frac{\l}{2} \sum \limits_{i=1}^{\infty} | \Q_{j_{i}} \cap \Omega|  = 
		\frac{\l}{2k^{Q+2}} \Big | \bigcup \limits_{i=1}^{\infty} k\Q_{j_{i}} \cap \Omega \Big |
		\ge \frac{\l}{2k^{Q+2}} \Big | \{ Mf < \l \} \cap \Omega \Big |.
	\end{align*} 
	Thus, the claim is proved when $C=2k^{Q+2}$.
\end{proof}	
The following generalized version of the Lebesgue differentiation theorem holds. 
\begin{theorem}[Genaralized Lebesgue Differentiation Theorem]
	\label{gen-leb}
	Let $f \in L^{1}(\O, dx \otimes dt)$, where $\O$ is an open subset of $\R^{N+1}$. Then for a.e. 
	$(x,t) \in \O$ 
	\begin{equation*}
		\lim \limits_{r \to 0^{+}} \frac{1}{|\Q_{r}(x,t)|} \int \limits_{\Q_{r}(x,t)} 
		| f(y,s) - f(x,t)| \, dy \, ds = 0.
	\end{equation*}
\end{theorem}
\noindent
Theorem \ref{gen-leb} is obtained from the following Lemma \ref{max-fun} exactly as in \cite[Theorem 2.5.1]{IS-fn} by considering Lemma \ref{max-fun}.

\subsection{Ink-spots theorem without time delay}
\begin{lemma}
	\label{ink-spot2}
	Let $E \subset F \subset \Q_{1}$ be two bounded measurable sets. We make the following assumptions for 
	some constant $\mu \in (0,1)$:
	\begin{itemize}
		\item $E < (1 - \mu) |\Q_{1}|$;
		\item if for any cylinder $\Q \subset \Q_{1}$ such that $|\Q \cap E| \ge (1-\mu)|\Q|$, then $\Q \subset F$.
	\end{itemize}
	Then $|E| \le (1 - c\mu) |F|$ for some universal constant $c$ only depending on $N$.
\end{lemma}
\begin{proof}
	Thanks to Theorem \ref{gen-leb}, for almost all points $z \in E$ there is some cylinder $\Q^{z}$
	containing $z$ such that $|\Q^{z} \cap E| \ge (1 - \mu) |\Q^{z}|$. Thus, for all Lebesgue points $z \in E$ we 
	choose a maximal cylinder $\Q^{z} \subset \Q_{1}$ that contains $z$ and such that $|\Q^{z} \cap E| \ge (1 - 
	\mu) |\Q^{z}|$. Here $\Q^{z}= \Q_{\overline r} (\overline x , \overline t)$ for some $\overline r > 0$ and
	$(\overline x , \overline t) \in \Q_{1}$. In particular, we have that $\Q^{z}$ differs from $\Q_{1}$ and 
	$\Q^{z} \subset F$ by our assumption.
	
	First of all we prove that $|\Q^{z} \cap E| = (1 - \mu) |\Q^{z}|$. 
	By contradiction, let us suppose that is not true.
	Then there exists $\d > 0$ small enough and $\overline \Q$ such that $\Q^{z} \subset \overline \Q 
	\subset (1 + \d) \Q^{z}$, $\overline \Q \subset \Q_{1}$ and $|\overline \Q \cap E| > (1 - \mu) |\Q^{z}|$, and
	this contradicts the maximality of the choice of $\Q^{z}$. 
	
	Then we recall that the family of cylinders $\{ \Q^{z} \}_{z \in E}$ covers the set $E$. Thanks to Lemma 
	\ref{vitali} and considering that $E$ is a bounded set, 
	we can extract a finite subfamily of non overlapping cylinders $\Q_{j} := \Q^{z_{j}}$ such that 
	$E \subset \cup_{j=1}^{n} k \Q_{j}$. Since $\Q_{j} \subset F$ and $|\Q_{j} \cap E| = (1 - \mu) |\Q_{j}|$, we 
	have that $|\Q_{j} \cap F \setminus E| = \mu |\Q^{z}|$. Therefore, 
	\begin{align*}
		|F \setminus E| \ge \sum \limits_{j=1}^{n} |\Q_{j} \cap F \setminus E| \ge 
		\sum \limits_{j=1}^{n} \mu |\Q_{j}| = k^{-(Q+2)} \mu  \sum \limits_{j=1}^{n} |k\Q_{j}|
		\ge k^{-(Q+2)} \mu  |E|.
	\end{align*}
	Thus, we get that $|F| \ge (1 + \overline c \mu) |E|$, with $\overline c = k^{-(Q+2)}$. Since $\overline c \mu \in 
	(0,1)$, we complete the proof by choosing $c = \overline c /2$. 	
\end{proof}

\subsection{Proof of Theorem \ref{ink-spot1} and Corollary \ref{ink-spot-cor}}
In order to proceed with the proof of the Ink Spots Theorem, we first need to recall two preliminary results. 
\begin{lemma}
	\label{leak1}
	Consider a (possibly infinite) sequence of intervals $(a_{j} - h_{k}, a_{j}]$. Then
	\begin{equation*}
		\Big | \bigcup \limits_{k} (a_{k}, a_{k} + mh_{k}] \Big | 
		\ge \frac{m}{m+1} \Big | \bigcup \limits_{k} (a_{k} - h_{k}, a_{k}] \Big | .
	\end{equation*}
\end{lemma}
\noindent
The proof of Lemma \ref{leak1} can be found in \cite[Lemma 10.8]{IS-weak}. 
Here, we only report the proof of the following lemma, that is an extension of Lemma 10.9 \cite{IS-weak}.
\begin{lemma}
	\label{leak2}
	Let $\{ \Q_{j} \}$ be a collection of slanted cylinders and let $\overline \Q_{j}^{m}$ be the corresponding versions as in 
	\eqref{qm}. Then 
	\begin{equation*}
		\Big | \bigcup \limits_{j}  \overline \Q_{j}^{m} \Big | \ge \frac{m}{m+1} 
		\Big | \bigcup \limits_{j}  \Q_{j} \Big |.
	\end{equation*}
\end{lemma}
\begin{proof}
	Because of Fubini's Theorem we know that for any set $\O \subset \R^{N+1}$
	\begin{equation*}
		|\O| = \int | \{ (x^{(1)}, \ldots, x^{(\k)},t) : (x^{(0)}, x^{(1)}, \ldots, x^{(\k)},t) \in \O \} | \, dx^{(0)}. 
	\end{equation*}
	Therefore, in order to prove our statement it is sufficient to show that for every $x^{(0)} \in \R^{m_{0}}$ 
	\begin{align*}
		\Big | \Big \{ (x^{(1)}, \ldots, x^{(\k)},t) &: (x^{(0)}, x^{(1)}, \ldots, x^{(\k)},t) \in \bigcup \limits_{j}  \overline 
		\Q_{j}^{m} \Big \} \Big | \\
		&\ge  \frac{m}{m+1} 
		\Big | \Big \{ (x^{(1)}, \ldots, x^{(\k)},t) : (x^{(0)}, x^{(1)}, \ldots, x^{(\k)},t) \in \bigcup \limits_{j}  \Q_{j} 
		\Big \} \Big |
	\end{align*}
	From now on, let us consider a fixed $\overline x \in \R^{m_{0}}$. 
	Any cylinder $\Q_{j}$ is a cylinder with center 
	\begin{equation*}
	(x_{j}^{(0)}, x_{j}^{(1)}, \ldots, x_{j}^{(\k)}, t_{j}) \in \R^{N+1}
	\end{equation*}
	and radius $r_{j} >0$. $\overline \Q_{j}^{m}$ is its delayed version \eqref{qm}, that
	thanks to Remark \ref{ball-cil} can equivalently be represented as follows
	\begin{equation*}
		\overline \Q^{m}_{j} = (t_{0}, t_{0} + mr_{j}^{2}] 
		\times B_{r}(x^{(0)}_{j}) \times B_{(m+2)r_{j}^{3}}(x^{(1)}_{j}) \times \ldots 
		\times B_{(m^{\k} + 2 \sum_{i=0}^{\k} m^{i})r_{j}^{2\k + 1}}(x^{(\k)}_{j}).
	\end{equation*}
	On one hand, when $|\overline x - x^{(0)}_{j}|\ge r_{j}$ the set 
	\begin{equation*}
		\Big \{ (x^{(1)}, \ldots, x^{(\k)},t) : (x^{(0)}, x^{(1)}, \ldots, x^{(\k)},t) \in \overline 
		\Q_{j}^{m} \Big \} \quad \text{is empty}.
	\end{equation*}
	On the other hand, when $|\overline x - x^{(0)}_{j}| < r_{j}$ we have that 
	\begin{align*}
		\Big \{ (x^{(1)}, \ldots, x^{(\k)},t) &: (\overline x, x^{(1)}, \ldots, x^{(\k)},t) \in \overline 
		\Q_{j}^{m} \Big \} \\
		&\supset  \, \widetilde \Q_{j} := \, 
		(t_{j}, t_{j} + mr^{2}_{j}] \times B_{2r_{j}^{3}}(x^{(1)}_{j}) \times \ldots 
		\times B_{2 \sum_{i=0}^{\k-1} m^{i}r_{j}^{2\k + 1}}(x^{(\k)}_{j}).
	\end{align*}
	Based on these last observations, we have that 
	\begin{align*}
		\Big | \Big \{ (x^{(1)}, \ldots, x^{(\k)},t) &: (x^{(0)}, x^{(1)}, \ldots, x^{(\k)},t) \in \bigcup \limits_{j}  \overline 
		\Q_{j}^{m} \Big \} \Big | \ge  
		\Big |  \bigcup \limits_{j: |\overline x - x^{(0)}_{j}|< r_{j} }\widetilde \Q_{j}
		\Big | .
	\end{align*}
	Now, thanks to Fubini's Theorem and Lemma \ref{leak1} we have 
	\begin{align*}
		\Big | \Big \{ (x^{(1)}, \ldots, x^{(\k)},t) &: (x^{(0)}, x^{(1)}, \ldots, x^{(\k)},t) \in \bigcup \limits_{j}  \overline 
		\Q_{j}^{m} \Big \} \Big | \\
		&\ge  \frac{m}{m+1}
		\Big |  \bigcup \limits_{j: |\overline x - x^{(0)}_{j}|< r_{j} } 
		(t_{j} - r^{2}_{j}, 0] \times B_{2r_{j}^{3}}(x^{(1)}_{j}) \times \ldots 
		\times B_{2 \sum_{i=0}^{\k-1} m^{i}r_{j}^{2\k + 1}}(x^{(\k)}_{j})
		\Big | \\
		&\ge \frac{m}{m+1}
		\Big |  \bigcup \limits_{j: |\overline x - x^{(0)}_{j}|< r_{j} } 
		(t_{j} - r^{2}_{j}, 0] \times B_{r_{j}^{3}}(x^{(1)}_{j}) \times \ldots 
		\times B_{r_{j}^{2\k + 1}}(x^{(\k)}_{j})
		\Big | \\
		&=  \frac{m}{m+1}
		\Big | \Big \{ (x^{(1)}, \ldots, x^{(\k)},t) : (\overline x, x^{(1)}, \ldots, x^{(\k)},t) \in \bigcup \limits_{j}
		\Q_{j} \Big \} \Big |.
	\end{align*}
	Combining all of the above results, the proof is complete. 
\end{proof}

\textit{Proof of Theorem \ref{ink-spot1}.} Let $Q$ be the collection of all cylinders $\Q \subset \Q_{1}$ 
such that $|\Q \cap E| \ge (1 - \mu) |\Q|$. Let $G := \cup_{\Q \in Q}\Q$. By construction, the sets $E$ and $G$ satisfy the assumptions of Lemma \ref{ink-spot2}. Therefore $(1 - c_{\rm is} \mu)|G| \ge |E|$. Combining the assumptions of the theorem with Lemma \ref{leak2} we conclude the proof. $\hfill \square$

\bigskip

\textit{Proof of Corollary \ref{ink-spot-cor}.} The condition $|E| \le (1-\d)|\Q_{1}|$ is implied by the second assumption when $r_{0} < 1$. Moreover, the result is trivial when $r_{0} \ge 1$ choosing $C$ sufficiently large.
Let $Q$ be the collection of all cylinders $\Q\subset \Q_{1}$ such that $|\Q \cap E| \ge (1 - \mu) |\Q|$. Let 
$G:= \cup_{\Q \in Q}\overline \Q^{m}$. From Theorem \ref{ink-spot1} we have that 
$|E| \le \frac{m}{m+1}(1 - c \mu) |G|$. Moreover, our assumptions tell us $G \subset F$. In order to conclude the proof is sufficient to estimate the measure $G \setminus \Q_{1}$ by considering that each
cylinder $\Q= \Q_{r} (x,t) \subset \Q_{1}$ has radius bounded by $r_{0}$ (see \cite[Corollary 10.2]{IS-weak}). 
$\hfill \square$

	\setcounter{theorem}{0}
	\counterwithin{theorem}{section}
	\counterwithin{equation}{section}
\section{Stacked cylinders}
\label{stacked}
For the sake of completeness, we here state the stacking cylinders lemma for our operator $\L$. Such a result is used when applying the Ink-Spots Theorem in the proof of the weak Harnack inequality, Theorem \ref{weak-harnack}.
\begin{lemma}
	\label{stackedcylinders}
	Let $\omega<\frac{1}{\sqrt{2\kappa+1}}$ and $\r=\left((3\kappa+1)\omega\right)^\frac{1}{2\kappa+1}$. We 
	consider any non-empty cylinder $\Q_r(z_0)\subset\Q_-$ and we set $T_k=\sum_{j=1}^k(2^kr)^2$. Let 
	$N\geq 1$ such that $T_N\leq-t_0<T_{N+1}$ and let
	\begin{align*}
		\Q_r[k]&:=\Q_{2^k r}(z_k),\quad k=1,\ldots,N\\
		\Q_{R_{N+1}}[N+1]&:=\Q_{R_{N+1}}(z_{N+1}),
	\end{align*}
	where $z_k=z_0\circ (0,\ldots,0,T_k)$ and $R=|t_0+T_N|^{\frac{1}{2}}$, $R_{N+1}=\max(R,\r)$, and
	\begin{equation*}
		z_{N+1}=\left\{ 
			\begin{array}{ll} 
				z_N\circ(0,\ldots,0,R^2),\quad &\textit{if $  R\geq \r$}\\
				(0,0),\quad &\textit{if $ R<\r$}
			\end{array} \right.
	\end{equation*}
	These cylinders satisfy
	\begin{eqnarray*}
		\Q_+\subset \Q_{R_{N+1}}[N+1],\qquad \cup_{k=1}^{N+1}\Q_r[k]\subset(-1,0]\times B_2,\qquad 
		\tilde{\Q}[N]\subset\Q_r[N],
	\end{eqnarray*}
	where $\tilde{\Q}[N]=\Q_{R_{N+1}/2}\left(z_{N+1}\circ (0,\ldots,0,-R_{N+1}^2)\right)$.
\end{lemma}
\begin{proof}
As our derivation of the previous lemma follows very closely the one contained in \cite[Appendix C]{GI}, we here do not write explicitly the proof. Indeed, the proof of the result is merely geometric and the main difference with \cite{GI} lies in the fact that in our case we exploit the more general composition law and dilations defined in \eqref{grouplaw} and in \eqref{gdil}, respectively. This explains why here the constants $\omega$ and $\r$ differ from the ones in \cite{GI}.  
\end{proof}

\end{document}